\documentclass[12pt, reqno]{amsart}
\usepackage{amssymb, amsthm, amsmath, amsfonts,mathtools}
\usepackage{array, epsfig}
\usepackage{bbm}

\usepackage{hyperref}
\usepackage[numbers,square]{natbib}
\usepackage{color}
\allowdisplaybreaks

  \evensidemargin
\oddsidemargin
\newcommand{\ii}{{\rm{i}}}

\newcommand{\bA}{\mathbb{A}}

\newcommand{\bB}{\mathbb{B}}

\newcommand{\bQ}{\mathbb{Q}}

\newcommand{\bS}{\mathbb{S}}
\newcommand{\bI}{\mathbb{I}}
\newcommand{\bJ}{\mathbb{J}}

\newcommand{\cF}{\mathcal{F}}

\def\dint{\textup{d}}
\newcommand{\Mod}[1]{\ (\mathrm{mod}\ #1)}

\newcommand{\E}{\mathbb E}

\newcommand{\R}{\mathbb{R}}
\newcommand{\N}{\mathbb{N}}

\renewcommand{\P}{\mathbb{P}}

\newcommand{\Vol}{\mathop{\mathrm{Vol}}\nolimits}

\newcommand{\eps}{\varepsilon}

\newcommand{\ind}{\mathbbm{1}}

\newcommand{\dd}{{\rm d}}
\newcommand{\eee}{{\rm e}}

\theoremstyle{plain}
\newtheorem{theorem}{Theorem}[section]
\newtheorem{lemma}[theorem]{Lemma}

\newtheorem{proposition}[theorem]{Proposition}
\newtheorem{conjecture}[theorem]{Conjecture}

\theoremstyle{definition}

\theoremstyle{remark}
\newtheorem{remark}[theorem]{Remark}

\newcommand{\stirling}[2]{\genfrac{[}{]}{0pt}{}{#1}{#2}}
\newcommand{\stirlingsec}[2]{\genfrac{\{}{\}}{0pt}{}{#1}{#2}}

\begin{document}

\author{Zakhar Kabluchko}
\address{Zakhar Kabluchko: Institut f\"ur Mathematische Stochastik,
Westf\"alische Wilhelms-Universit\"at M\"unster,
Orl\'eans-Ring 10,
48149 M\"unster, Germany}
\email{zakhar.kabluchko@uni-muenster.de}

\title[Angles of random simplices]{Recursive scheme for angles of random simplices, \\and applications to random polytopes}



\begin{abstract}
Consider a random simplex $[X_1,\ldots,X_n]$ defined as the convex hull of independent identically distributed (i.i.d.)\ random points $X_1,\ldots,X_n$ in $\mathbb{R}^{n-1}$ with the following beta density:
$$
f_{n-1,\beta} (x) \propto (1-\|x\|^2)^{\beta} \ind_{\{\|x\| < 1\}},
\qquad
x\in\mathbb{R}^{n-1},
\quad
\beta>-1.
$$
Let $J_{n,k}(\beta)$ be the expected internal angle of the simplex $[X_1,\ldots,X_n]$ at its face $[X_1,\ldots,X_k]$. Define $\tilde J_{n,k}(\beta)$ analogously for i.i.d.\ random points distributed according to the beta' density
$$
\tilde f_{n-1,\beta} (x) \propto (1+\|x\|^2)^{-\beta},
\qquad
x\in\mathbb{R}^{n-1},
\quad
\beta > \frac{n-1}{2}.
$$
We derive formulae for $J_{n,k}(\beta)$ and $\tilde J_{n,k}(\beta)$ which make it possible to compute these quantities symbolically, in finitely many steps,  for any integer or half-integer value of $\beta$. For $J_{n,1}(\pm 1/2)$ we even provide explicit formulae in terms of products of Gamma functions.
We give applications of these results to two seemingly unrelated problems of stochastic geometry.

\vspace*{1mm}
\noindent
(i) We compute explicitly the expected $f$-vectors of the typical Poisson-Voronoi cells in dimensions up to $10$. %We also study Poisson-Voronoi tessellations of the sphere and compute the expected $f$-vectors of typical cells there.

\vspace*{1mm}
\noindent
(ii) Consider the random polytope $K_{n,d} := [U_1,\ldots,U_n]$ where $U_1,\ldots,U_n$ are i.i.d.\ random points sampled uniformly inside some $d$-dimensional convex body $K$ with smooth boundary and unit volume. M.~Reitzner [\textit{Adv.\ Math.}, 2005] proved the existence of the limit of the normalized expected $f$-vector of $K_{n,d}$:
$$
\lim_{n\to\infty}
n^{-{\frac{d-1}{d+1}}}\E \mathbf f(K_{n,d}) = \mathbf c_d \cdot \Omega(K),
$$
where $\Omega(K)$ is the affine surface area of $K$, and $\mathbf c_d$ is an unknown vector not depending on $K$.  We compute $\mathbf c_d$
explicitly in dimensions up to $d=10$ and also solve the analogous problem for random polytopes with vertices distributed uniformly on the sphere.

%\vspace*{1mm}
%We conjecture a reciprocity law relating expected internal angles to expected external angles  and provide evidence for it.
%and, in particular, a ``nice'' formula for the expected internal angles generalizing the corresponding formula for the regular simplex.
\end{abstract}

\maketitle
\section{Statement of the problem}
\subsection{Introduction}
It is well known that the sum of angles in any plane triangle is constant, whereas the sum of solid $d$-dimensional angles at the vertices of a $d$-dimensional simplex is not, starting with dimension $d=3$. It is therefore natural to ask what is the ``average'' angle-sum of a $d$-dimensional simplex. To define the notion of average, we put a probability measure on the set simplices as follows.
%Of course, the notion of average depends on the probability distribution we put on the set of all simplices.
Let $X_1,\ldots,X_n$ be independent, identically distributed (i.i.d.)\ random points in $\R^{n-1}$ with probability distribution $\mu$. Consider a random simplex defined as their convex hull:
$$
[X_1,\ldots,X_n] := \{\lambda_1X_1+\ldots+\lambda_n X_n\colon \lambda_1+\ldots+\lambda_{n}=1, \lambda_1\geq 0,\ldots, \lambda_{n}\geq 0\}.
$$
For the class of distributions studied here, this simplex is non-degenerate (i.e., has a non-empty interior) a.s. Let $\beta ([X_1,\ldots,X_{k}], [X_1,\ldots,X_{n}])$ denote the internal angle of the simplex $[X_1,\ldots,X_n]$ at its $(k-1)$-dimensional face $[X_1,\ldots,X_k]$. Similarly, we denote by $\gamma([X_1,\ldots,X_{k}], [X_1,\ldots,X_{n}])$ the external (or normal) angle of  $[X_1,\ldots,X_n]$ at $[X_1,\ldots,X_k]$. The exact definitions of internal and external angles will be recalled in Section~\ref{sec:notation}; see also the book~\cite{SW08} for an extensive account of stochastic geometry.  We agree to choose the units of measurement for angles in such a way that the full-space angle equals $1$.
%Let $s_k([X_0,\ldots,X_d])$ be the sum of internal angles of $[X_0,\ldots,X_d]$ at all of its $k$-dimensional faces,  for $k\in \{0,\ldots,d-1\}$.
We shall be interested in the expected values of the above-defined angles. The special case when $\mu$ is a multivariate normal distribution has been studied in~\cite{kabluchko_zaporozhets_gauss_simplex}, \cite{goetze_kabluchko_zaporozhets} and~\cite{godland_kabluchko_zaporozhets}, where the following theorem has been demonstrated.

\begin{theorem}
\label{theo:gauss_simplex}
If $X_1,\ldots,X_n$ are i.i.d.\ random points in $\R^{n-1}$ having a non-degenerate multivariate Gaussian distribution, then the expected internal angle of $[X_1,\ldots,X_n]$ at the $k$-vertex face $[X_1,\ldots,X_k]$ coincides with the internal angle of  the regular $(n-1)$-dimensional simplex $[e_1,\ldots,e_n]$ at its face $[e_1,\ldots,e_k]$, for all $k\in \{1,\ldots,n\}$. Here, $e_1,\ldots,e_n$ denotes the standard orthonormal basis of $\R^n$.  The statement remains true if internal angles are replaced by the external ones.\footnote{In~\cite{kabluchko_zaporozhets_gauss_simplex}, the theorem has been established by two different methods for internal angles and only in the special case when $k=1$ and the Gaussian distribution is isotropic. The same proofs apply to arbitrary $k$'s. The full proof of Theorem~\ref{theo:gauss_simplex} in the isotropic case can be found in~\cite[Theorem~4.1]{goetze_kabluchko_zaporozhets}. The case of the non-isotropic Gaussian distribution has been settled in~\cite[Theorem 4.17]{godland_kabluchko_zaporozhets}. The Gaussian  simplex can be viewed as the limiting case of the so-called beta simplex as $\beta\to +  \infty$. For beta simplices and polytopes, results closely related to Theorem~\ref{theo:gauss_simplex} can be found in~\cite[Theorems~1.6, 1.12, 1.13]{beta_polytopes}.
%(restated as Theorem~\ref{theo:external} below) by letting $\beta\to +\infty$ there and observing that the result coincides with the known formula for external angles of regular simplices.
}
\end{theorem}

\subsection{Beta and beta' distributions}
In the present paper we shall be interested in the case when $\mu$ belongs to one of the following two  remarkable families of probability distributions introduced by Miles~\cite{miles} and studied by Ruben and Miles~\cite{ruben_miles}.
%see also~\cite{beta_polytopes_temesvari,beta_simplices,beta_polytopes}.
A random vector in $\R^d$ has a $d$-dimensional \textit{beta distribution} with parameter $\beta>-1$ if its Lebesgue density is
\begin{equation}\label{eq:def_f_beta}
f_{d,\beta}(x)=c_{d,\beta} \left( 1-\left\| x \right\|^2 \right)^\beta\ind_{\{\|x\| <  1\}},\qquad x\in\R^d,\qquad
c_{d,\beta}= \frac{ \Gamma\left( \frac{d}{2} + \beta + 1 \right) }{ \pi^{ \frac{d}{2} } \Gamma\left( \beta+1 \right) }.
\end{equation}
Here, $\|x\| = (x_1^2+\ldots+x_d^2)^{1/2}$ denotes the Euclidean norm of the vector $x= (x_1,\ldots,x_d)\in\R^d$. Similarly, a random vector in $\R^d$ has \textit{beta' distribution} with parameter $\beta>d/2$ if its Lebesgue density is given by
\begin{equation}\label{eq:def_f_beta_prime}
\tilde{f}_{d,\beta}(x)=\tilde{c}_{d,\beta} \left( 1+\left\| x \right\|^2 \right)^{-\beta},\qquad
x\in\R^d,\qquad
\tilde{c}_{d,\beta}= \frac{ \Gamma\left( \beta \right) }{\pi^{ \frac{d}{2} } \Gamma\left( \beta - \frac{d}{2} \right)}.
\end{equation}
The following particular cases are of special interest:
\begin{itemize}
\item[(a)] The beta distribution with $\beta=0$ is the uniform distribution in the unit ball $\bB^{d}:=\{x\in\R^d\colon \|x\| \leq 1\}$.
\item[(b)] The weak limit of the beta distribution as $\beta \downarrow -1$ is the uniform distribution on the unit sphere $\bS^{d-1} := \{x\in \R^d\colon \|x\| = 1\}$; see~\cite{beta_polytopes_temesvari}. In the following, we write $f_{d,-1}$ for the uniform distribution on $\bS^{d-1}$, and the results of the present paper apply to the case $\beta=-1$.
\item[(c)]  The standard normal distribution on $\R^d$ is the weak limit of both  beta and beta' distributions (after suitable rescaling) as $\beta\to +\infty$; see~\cite[Lemma~1.1]{beta_polytopes}.
\item[(d)] The beta' distribution $\tilde f_{n-1,n/2}$ on $\R^{n-1}$ with $\beta = n/2$ is the image of the uniform distribution on the upper half-sphere $\bS^{n-1}_+$ under the so-called gnomonic projection~\cite[Proposition~2.2]{convex_hull_sphere}; see also~\cite{kabluchko_poisson_zero} for further applications of this observation.
\end{itemize}

\subsection{Expected internal angles}
Let $X_1,\ldots,X_{n}$ be independent random points in $\R^{n-1}$ distributed according to the beta distribution $f_{n-1,\beta}$, where $\beta\geq -1$. Their convex hull $[X_1,\ldots,X_n]$ is called the $(n-1)$-dimensional \textit{beta simplex}.
We shall be interested in the expected internal angles of these random simplices, denoted by
$$
J_{n,k}(\beta) := \E \beta ([X_1,\ldots,X_{k}], [X_1,\ldots,X_{n}]),
$$
for all $n\in\N$ and $k\in \{1,\ldots,n\}$. By definition, $J_{n,n}(\beta) = 1$ for all $n\in\N$.
Similarly, let $\tilde X_1,\ldots,\tilde X_{n}$ be independent random points in $\R^{n-1}$ distributed according to the beta' distribution $\tilde f_{n-1,\beta}$, where $\beta> (n-1)/2$. Their convex hull $[\tilde X_1,\ldots,\tilde X_n]$ is called the $(n-1)$-dimensional \textit{beta' simplex} and its expected internal angles are denoted by
$$
\tilde J_{n,k}(\beta) := \E \beta ([\tilde X_1,\ldots,\tilde X_{k}], [\tilde X_1,\ldots,\tilde X_{n}],
$$
for all $n\in\N$ and $k\in \{1,\ldots,n\}$. Again, we define $\tilde J_{n,n}(\beta) = 1$ for all $n\in\N$.
Note that the subscripts $n$, respectively $k$,  refer to the number of vertices of the simplex, respectively, of the face of interest, rather than to the corresponding dimensions.  By exchangeability, for both beta and beta' simplices, it does not matter which face with $k$ vertices to take. Hence, the expected sum of internal angles at all $k$-vertex faces of the corresponding simplex is
$$
\bJ_{n,k}(\beta) :=  \binom {n}{k}  J_{n,k}(\beta),
\qquad
\tilde \bJ_{n,k}(\beta) :=  \binom {n}{k}  \tilde J_{n,k}(\beta).
$$
The triangular arrays $J_{n,k}(\beta)$ and $\tilde J_{n,k}(\beta)$ appeared in~\cite{beta_polytopes} together with the closely related arrays $I_{n,k}(\alpha)$ and $\tilde I_{n,k}(\alpha)$ that are essentially the expected external angles of beta and beta' simplices; see Theorems~\ref{theo:external} and~\ref{theo:external_prime}, below. It has been shown in~\cite{beta_polytopes} that many quantities appearing in stochastic geometry can be expressed in terms of $I_{n,k}(\beta)$, $\tilde I_{n,k}(\beta)$ and $J_{n,k}(\beta)$, $\tilde J_{n,k}(\beta)$. An incomplete list of such quantities is as follows:
\begin{itemize}
\item[(a)] The expected $f$-vectors of beta- and beta' polytopes. The beta polytopes are defined as random polytopes of the form $P_{n,d}^{\beta}:=[Z_1,\ldots,Z_n]$, where $Z_1,\ldots,Z_n$ are i.i.d.\ random points in $\R^d$ with distribution of the form $f_{d,\beta}$. The beta' polytope $\tilde P_{n,d}^\beta$ is defined similarly.% upon replacing the underlying density by $\tilde f_{d,\beta}$.
    %These are defined as random polytopes of the form $[X_1,\ldots,X_n]$, respectively $[\tilde X_1,\ldots,\tilde X_n]$, where $X_1,\ldots,X_n$, respectively, $\tilde X_1,\ldots,\tilde X_n$,  are i.i.d.\ random points in $\R^d$ with distribution of the form $f_{d,\beta}$, respectively, $\tilde f_{d,\beta}$.
\item[(b)] Expected internal and external angles of beta and beta' polytopes, and, more generally, expected intrinsic conic volumes of their tangent cones.
\item[(c)] Expected $f$-vector of the zero cell of the Poisson hyperplane tessellation and expected $f$-vectors of the random polytopes in the half-sphere; see~\cite{kabluchko_poisson_zero} for a detailed study of these models.
\item[(d)] Expected $f$-vector of the typical Poisson-Voronoi cell.
\item[(e)] Constants appearing in the work of Reitzner~\cite{ReitznerCombinatorialStructure} on the asymptotics of the expected $f$-vectors of random polytopes approximating smooth convex bodies.
\item[(f)] External and internal angles of the regular simplex with $n$ vertices at its $k$-vertex faces. These coincide with the corresponding expected angles of the random Gaussian simplex~\cite{kabluchko_zaporozhets_gauss_simplex,goetze_kabluchko_zaporozhets}, and are given by $I_{n,k}(+\infty):= \lim_{\beta\uparrow +\infty} I_{n,k}(\beta)$ and $J_{n,k}(+\infty) := \lim_{\beta\uparrow +\infty} J_{n,k}(\beta)$, respectively.
\end{itemize}

While there exist explicit formulae for $I_{n,k}(\alpha)$ and $\tilde I_{n,k}(\alpha)$ (see Section~\ref{subsec:external}), no general formulae are known for $\bJ_{n,k}(\beta)$ and $\tilde \bJ_{n,k}(\beta)$ except in some special cases.  For example, we have $\bJ_{3,1} (\beta) = 1/2$ because the sum of angles in any plane triangle equals half the full angle.
%$\bJ_{3,2} (\beta) = 3/2$ and $\bJ_{3,3} (\beta) = 1$ for all $\beta\geq -1$.
For general $n\in\N$, it always holds that $\bJ_{n,n}(\beta) = 1$ and $\bJ_{n,n-1}(\beta) = n/2$, and all these formulae are valid in the beta' case, too.  A general combinatorial formula for $\tilde \bJ_{n, k}(n/2)$ was derived in~\cite{kabluchko_poisson_zero}, where it was used to compute the expected $f$-vector of the Poisson zero polytope. For $n=4$ and $n=5$, explicit formulae for $\bJ_{n,k}(\beta)$ were derived in~\cite{kabluchko_angles} by a method not allowing for an extension to higher dimensions. The main results of the present paper can be summarized as follows.    In Section~\ref{sec:main_results}, we derive a formula which enables us to compute $\bJ_{n,k}(\beta)$ and $\tilde \bJ_{n,k}(\beta)$ symbolically for half-integer $\beta$, and numerically for all admissible $\beta$. The main work for this formula has been done in~\cite{beta_polytopes} and~\cite{kabluchko_poisson_zero}, while the main aim of the present paper is to state it explicitly and to demonstrate its consequences. The latter will be done in Section~\ref{sec:applications}, where we apply the formula to compute (among other examples) the expected $f$-vectors of typical Poisson-Voronoi cells and the constants that appeared in the work of Reitzner~\cite{ReitznerCombinatorialStructure} on random polytopes approximating convex bodies, in  dimensions  up to $10$.

%For a polytope $P$ and its face $F$, denote by $\beta(F,P)$ the internal angle of $P$ at $F$ normalized in such a way that the full-space angle is $1$.

\subsection{Expected external angles}\label{subsec:external}
The  following two theorems define the quantities $I_{n,k}(\alpha)$ and $\tilde I_{n,k}(\alpha)$ and relate them to the expected external angles of beta and beta' simplices.  They are special cases of Theorems 1.6 and 1.16 in~\cite{beta_polytopes}, respectively.
\begin{theorem}\label{theo:external}
Let $X_1,\ldots,X_n$ be i.i.d.\ random points in $\R^{n-1}$ with beta density $f_{n-1,\beta}$, $\beta\geq -1$ (which is interpreted as the uniform distribution on the sphere $\bS^{n-2}$ if $\beta=-1$). Then, for all $k\in \{1,\ldots,n\}$, the expected external angle of the beta simplex $[X_1,\ldots,X_n]$ at its face $[X_1,\ldots,X_k]$ is given by
$$
\E \gamma ([X_1,\ldots,X_k], [X_1,\ldots,X_n]) = I_{n,k}(2\beta + n-1),
$$
where for $\alpha>-1/k$ we define
\begin{equation}\label{eq:I_n_k}
I_{n,k}(\alpha)
=\int_{-\pi/2}^{+\pi/2} c_{1,\frac{\alpha k - 1}{2}} (\cos \varphi)^{\alpha k} \left(\int_{-\pi/2}^\varphi c_{1,\frac{\alpha-1}{2}}(\cos \theta)^{\alpha} \,\dd \theta \right)^{n-k} \, \dd \varphi.
\end{equation}
\end{theorem}
\begin{theorem}\label{theo:external_prime}
Let $\tilde X_1,\ldots,\tilde X_n$ be i.i.d.\ random points in $\R^{n-1}$ with beta' density $\tilde f_{n-1,\beta}$, where $\beta > \frac {n-1}{2}$.  Then, for all $k\in \{1,\ldots,n\}$,  the expected external angle of the beta' simplex $[\tilde X_1,\ldots,\tilde X_n]$ at its face $[\tilde X_1,\ldots,\tilde X_k]$ is given by
$$
\E \gamma ([\tilde X_1,\ldots,\tilde X_k], [\tilde X_1,\ldots,\tilde X_n]) = \tilde I_{n,k}(2\beta - n + 1),
$$
where for $\alpha>0$ we define
\begin{equation}\label{eq:I_n_k_tilde}
\tilde I_{n,k}(\alpha)
=\int_{-\pi/2}^{+\pi/2} \tilde c_{1,\frac{\alpha k + 1}{2}} (\cos \varphi)^{\alpha k-1} \left(\int_{-\pi/2}^\varphi \tilde c_{1,\frac{\alpha+1}{2}}(\cos \theta)^{\alpha-1} \,\dd \theta \right)^{n-k} \, \dd \varphi.
\end{equation}
\end{theorem}
Usually, it will be more convenient to work with angle sums rather than with  individual angles, which is why we introduce the quantities
\begin{equation}
\bI_{n,k}(\alpha) := \binom nk I_{n,k}(\alpha),
\qquad
\tilde \bI_{n,k}(\alpha) := \binom nk \tilde I_{n,k}(\alpha).
\end{equation}
Note that $\bI_{n,n}(\alpha) = \tilde \bI_{n,n}(\alpha) = 1$.

\section{Main results}\label{sec:main_results}

%\subsection{Algorithm for computing the quantities \texorpdfstring{$\bJ_{n,k}(\beta)$}{J\_\{n,k\}(beta)} and \texorpdfstring{$\tilde \bJ_{n,k}(\beta)$}{tilde J\_\{n,k\}(beta)}}
\subsection{Algorithm for computing expected internal-angle sums}\label{sec:relations}
In the next proposition we state relations which enable us to express the quantities $\bJ_{n,k}(\beta)$ through the quantities $\bI_{n,k}(\alpha)$. The proof will be given in Section~\ref{subsec:proof_relations}, where we shall also discuss the similarity between these relations and McMullen's angle-sum relations~\cite{mcmullen,mcmullen_polyhedra} for deterministic polytopes.
%Since the latter quantities can be computed symbolically for each integer $\alpha>0$

\begin{proposition}\label{prop:relations}
For every $n\in \{2,3,\ldots\}$, $k\in \{1,\ldots,n-1\}$ and every $\beta\geq -1$  the following relations between the quantities $\bI_{n,m}(\alpha)$ and $\bJ_{m,k}(\beta)$ hold:
\begin{align}
&
\sum_{\substack{s=0,1,\ldots\\ n-s\geq k}}  \bI_{n,n-s}(2\beta+n-1)  \bJ_{n-s,k} \left(\beta+\frac s2\right) = \binom nk,\label{eq:relation_I_J_1}\\
%\sum_{m=k}^n  \bI_{n,m}(2\beta+n-1)  \bJ_{m,k} \left(\beta+\frac {n-m}2\right) = \binom nk,\label{eq:relation_I_J_1}\\
&
\sum_{\substack{s=0,1,\ldots\\ n-s\geq k}}  (-1)^s\bI_{n,n-s}(2\beta+n-1)  \bJ_{n-s,k} \left(\beta+\frac s2\right) = 0.
%\sum_{m=k}^n  (-1)^{m} \bI_{n,n-s}(2\beta+n-1)  \bJ_{m,k} \left(\beta+\frac {n-m}2\right) = 0.
\label{eq:relation_I_J_2}
\end{align}
Similarly, for every $n\in \{2,3,\ldots\}$, $k\in \{1,\ldots,n-1\}$ and for every $\beta>(n-1)/2$, the quantities $\tilde \bI_{n,m}(\alpha)$ and $\tilde \bJ_{m,k}(\beta)$ satisfy the following relations:
\begin{align}
&
%\sum_{\substack{s=0,1,\ldots\\ n-2s\geq k}} \tilde \bI_{n,n-2s}(2\beta-n+1) \tilde \bJ_{n-2s,k} (\beta-s) = \frac 12 \binom nk,
\sum_{\substack{s=0,1,\ldots\\ n-s\geq k}} \tilde \bI_{n,n-s}(2\beta-n+1) \tilde \bJ_{n-s,k} \left(\beta-\frac s2\right) =  \binom nk,
\label{eq:relation_I_J_1_tilde}\\
&
%\sum_{\substack{s=0,1,\ldots\\ n-2s-1\geq k}} \tilde \bI_{n,n-2s-1}(2\beta-n+1) \tilde \bJ_{n-2s-1,k} \left(\beta - s - \frac 12\right) = \frac 12 \binom nk.
\sum_{\substack{s=0,1,\ldots\\ n-s\geq k}} (-1)^s\tilde \bI_{n,n-s}(2\beta-n+1) \tilde \bJ_{n-s,k} \left(\beta-\frac s2\right) =  0.
\label{eq:relation_I_J_2_tilde}
\end{align}
\end{proposition}

We now explain how these relations can be used to compute the quantities $\bJ_{n,k}(\beta)$ and $\tilde \bJ_{n,k}(\beta)$. Since the results in these two cases are similar to each other, we restrict ourselves to $\bJ_{n,k}(\beta)$ and state the results for $\tilde \bJ_{n,k}(\beta)$ at the end of the section.
First of all, we have $\bJ_{1,1}(\beta) = 1$. Assume that for some $n\in \{2,3,\ldots\}$ we are able to compute (symbolically or numerically) the quantities $\bJ_{m,k}(\gamma)$ with arbitrary $m\in \{1,\ldots,n-1\}$, $k\in \{1,\ldots,m\}$, and $\gamma\geq -1/2$. We are going to compute the quantities $\bJ_{n,k}(\beta)$ with $k\in \{1,\ldots,n\}$ and $\beta\geq -1$. If $k=n$, we trivially have $\bJ_{n,n}(\beta) = 1$. For  $k\in \{1,\ldots,n-1\}$ we use the formula
\begin{equation}\label{eq:alg_J_1}
\bJ_{n,k}(\beta)
=
\binom {n}{k} - \sum_{s=1}^{n-k}  \bI_{n,n-s}(2\beta+n-1) \bJ_{n-s,k}\left(\beta+\frac s2\right),
\end{equation}
which follows from~\eqref{eq:relation_I_J_1} by separating the term with $s=0$. Note that on the right-hand side we have the quantities of the type $\bI_{n,n-s}(\gamma)$ (which are just trigonometric integrals; see Section~\ref{subsec:external}) and the quantities $\bJ_{n-s,k}(\beta+\frac s 2)$ which  are already assumed to be known by the induction assumption since $n-s < n$.

The above recursive procedure allows us to express $\bJ_{n,k}(\beta)$ as a polynomial in the variables $\bI_{m,\ell}(2\beta+n-1)$ with $1\leq \ell < m \leq n$. Note that all terms have the same $\beta$-parameter $2\beta+n-1$. For example, for $n=4$ we obtain
\begin{align*}
\bJ_{4,1}(\beta)
&=
3 -2 \bI_{4,3}(3 + 2 \beta) - \bI_{4,2}(3 + 2 \beta)  + \bI_{4,3}(3 + 2 \beta) \bI_{3,2}(3 + 2 \beta),\\
\bJ_{4,2}(\beta)
&=
6 -3 \bI_{4,3}(3 + 2 \beta)  - \bI_{4,2}(3 + 2 \beta) + \bI_{4,3}(3 + 2 \beta) \bI_{3,2}(3 + 2 \beta),\\
\bJ_{4,3}(\beta) &=  4-\bI_{4,3}(3 + 2 \beta),\\
\bJ_{4,4}(\beta) &=1.
\end{align*}
We simplified the first line by using that $\bI_{n,1}(\alpha)=1$. Also, note that, in fact, $\bI_{3,2}(\alpha) = 3/2$ and $\bI_{4,3}(\alpha) = 2$.  More generally, we shall prove the following
\begin{theorem}\label{theo:formula_J_1}
For every $\beta\geq -1$, $n\in\N$ and $k\in \{1,\ldots,n\}$ we have
$$
\bJ_{n,k}(\beta) = \sum_{\ell=0}^{n-k} (-1)^\ell \sum \bI_{n, n_1}(2\beta+n-1) \bI_{n_1,n_2} (2\beta + n-1) \ldots \bI_{n_{\ell-1}, n_\ell}(2\beta + n-1) \binom {n_\ell}{k},
$$
where the second sum is taken over all integer tuples $(n_0, n_1,\ldots,n_{\ell})$ such that $n=n_0>n_1>\ldots>n_\ell\geq k$.
\end{theorem}

The following equation, which follows from~\eqref{eq:relation_I_J_1} and~\eqref{eq:relation_I_J_2} by taking the arithmetic mean, is more efficient for computational purposes since it contains less terms than~\eqref{eq:alg_J_1}:
\begin{equation}\label{eq:alg_J_2}
\bJ_{n,k}(\beta)
=
\frac 12 \binom {n}{k} - \sum_{s=1}^{\lfloor \frac{n-k}{2}\rfloor}  \bI_{n,n-2s}(2\beta+n-1) \bJ_{n-2s,k}(\beta+s).
\end{equation}
%By comparing the results of both algorithms one can obtain lots of non-trivial polynomial relations among the quantities $\bI_{n,k}(\beta)$.
For example, the first few non-trivial values of the internal-angles vector
$$
\bJ_{n,\bullet}(\beta):= (\bJ_{n,1}(\beta), \ldots, \bJ_{n,n}(\beta))
$$
are given by
\begin{flalign*}
\bJ_{4,\bullet}(\beta)
&=
(2-\bI_{4,2}(2 \beta+3), 3-\bI_{4,2}(2 \beta+3),2,1),&\\
\bJ_{5,\bullet}(\beta)
&=
\Bigg(\frac{3}{2}-\frac{1}{2} \bI_{5,3}(2 \beta+4), 5-\frac{3}{2} \bI_{5,3}(2 \beta+4),5-\bI_{5,3}(2 \beta+4),\frac{5}{2},1\Bigg),&\\
\bJ_{6,\bullet}(\beta)
&=
\Bigg(3-\bI_{6,2}(2 \beta+5)+\bI_{4,2}(2 \beta+5) \bI_{6,4}(2 \beta+5)-2 \bI_{6,4}(2 \beta+5),&\\
&\phantom{=\Bigg(}\frac{15}{2} -\bI_{6,2}(2 \beta+5)+\bI_{4,2}(2 \beta+5) \bI_{6,4}(2 \beta+5)-3 \bI_{6,4}(2 \beta+5),&\\
&\phantom{=\Bigg(}  10-2 \bI_{6,4}(2 \beta+5),\frac{15}{2}-\bI_{6,4}(2 \beta+5),3,1\Bigg)
.&
\end{flalign*}
Generalizing these formulae, we can prove the following
\begin{theorem}\label{theo:formula_J_2}
For every $\beta\geq -1$, $n\in\N$ and $k\in \{1,\ldots,n\}$ we have
$$
2 \bJ_{n,k}(\beta) -\delta_{n,k}
=
\sum_{\ell=0}^{\lfloor \frac{n-k}{2}\rfloor} (-1)^\ell \sum \bI_{n, n_1}(2\beta+n-1) \bI_{n_1,n_2} (2\beta + n-1) \ldots \bI_{n_{\ell-1}, n_\ell}(2\beta + n-1) \binom {n_\ell}{k},
$$
where $\delta_{n,k}$ is Kronecker's delta, and the sum is taken over all integer tuples $(n_0,n_1,\ldots,n_\ell)$ such that  $n=n_0>n_1>\ldots>n_\ell\geq k$ and such that $n-n_i$ is even for all $i\in\{1,\ldots,\ell\}$.
\end{theorem}

The quantities $\tilde \bJ_{n,k}(\beta)$ can be computed in a similar manner. We put  $\tilde \bJ_{1,1}(\beta) = 1$ and then use the recursive formula
$$
\tilde \bJ_{n,k}(\beta)
=
\binom {n}{k} - \sum_{s=1}^{n-k} \tilde \bI_{n,n-s}(2\beta-n+1)\tilde \bJ_{n-s,k}\left(\beta-\frac s2\right)
$$
that follows from~\eqref{eq:relation_I_J_1_tilde}.
Alternatively, one can use the more efficient formula
$$
\tilde \bJ_{n,k}(\beta)
=
\frac 12 \binom {n}{k} - \sum_{s=1}^{\lfloor \frac{n-k}{2}\rfloor} \tilde \bI_{n,n-2s}(2\beta-n+1)\tilde \bJ_{n-2s,k}(\beta-s)
$$
which follows from~\eqref{eq:relation_I_J_1_tilde} and~\eqref{eq:relation_I_J_2_tilde} by taking their arithmetic mean.
The next two theorems are similar to Theorems~\ref{theo:formula_J_1} and~\ref{theo:formula_J_2}. We omit their straightforward proofs.
\begin{theorem}\label{theo:formula_J_1_tilde}
For every $\beta > (n-1)/2$, $n\in\N$ and $k\in \{1,\ldots,n\}$ we have
$$
\tilde \bJ_{n,k}(\beta) = \sum_{\ell=0}^{n-k} (-1)^\ell \sum_{n=n_0>n_1>\ldots>n_\ell\geq k} \tilde \bI_{n, n_1}(2\beta - n + 1) \tilde \bI_{n_1,n_2} (2\beta - n + 1) \ldots \tilde \bI_{n_{\ell-1}, n_\ell}(2\beta - n + 1) \binom {n_\ell}{k},
$$
where the second sum is taken over all integer tuples $(n_0, n_1,\ldots,n_{\ell})$ such that $n=n_0>n_1>\ldots>n_\ell\geq k$.
\end{theorem}

\begin{theorem}\label{theo:formula_J_2_tilde}
For every $\beta > (n-1)/2$, $n\in\N$ and $k\in \{1,\ldots,n\}$ we have
$$
2\tilde \bJ_{n,k}(\beta) - \delta_{n,k} = \sum_{\ell=0}^{\lfloor \frac{n-k}{2}\rfloor} (-1)^\ell \sum \tilde \bI_{n, n_1}(2\beta - n + 1) \tilde \bI_{n_1,n_2} (2\beta - n + 1) \ldots \tilde \bI_{n_{\ell-1}, n_\ell}(2\beta - n + 1) \binom {n_\ell}{k},
$$
where $\delta_{n,k}$ is Kronecker's delta, and the sum is taken over all integer tuples $(n_0,n_1,\ldots,n_\ell)$ such that  $n=n_0>n_1>\ldots>n_\ell\geq k$ and such that $n-n_i$ is even for all $i\in\{1,\ldots,\ell\}$.
\end{theorem}

\subsection{Relations in matrix form}
Let us write Relation~\eqref{eq:relation_I_J_2}  in the following form:
$$
\sum_{m=k}^n  (-1)^{n-m} \bI_{n,m}(2\beta+n-1)  \bJ_{m,k} \left(\beta+\frac {n-m}2\right) = \delta_{nk},
$$
where $\delta_{nk}$ denotes the Kronecker delta. Introducing the new variable $\gamma:=\beta + (n-1)/2$ that ranges in the interval $[\frac{n-3}{2},+\infty)$, we can write
\begin{equation}\label{eq:relation_I_J_matrix_form}
\sum_{m=k}^n  (-1)^{n-m} \bI_{n,m}(2\gamma)  \bJ_{m,k} \left(\gamma-\frac {m-1}2\right) = \delta_{nk}.
\end{equation}
This relation has the advantage that now the $\bJ$-term does not contain $n$, which allows to state it in matrix form. Take some $N\in\N$, $\gamma\geq \frac{N-3}{2}$, and introduce the $N\times N$-matrices $\bA$ and $\bB$ with the following entries
\begin{align*}
\bA_{n,m}
&:=
\begin{cases}
(-1)^{n} \bI_{n,m}(2\gamma), &\text{ if } 1\leq m \leq n \leq N,\\
0, &\text{ otherwise},
\end{cases}\\
\bB_{m,k}
&:=
\begin{cases}
(-1)^{m} \bJ_{m,k}\left(\gamma - \frac{m-1}{2}\right), &\text{ if } 1\leq k \leq m \leq N,\\
0, &\text{ otherwise}.
\end{cases}
\end{align*}
Note that both $\bA$ and $\bB$ are lower-triangular matrices with $1$'s on the diagonal. Then, Relation~\eqref{eq:relation_I_J_matrix_form} states that $\bA \bB = E$, where $E$ is the $N\times N$-identity matrix. Since this implies that $\bB \bA = E$, we arrive at the following relation which is dual to~\eqref{eq:relation_I_J_matrix_form}:
\begin{equation}\label{eq:relation_I_J_matrix_form_dual}
\sum_{m=k}^n  (-1)^{n-m} \bJ_{n,m}\left(\gamma-\frac {n-1}2\right)  \bI_{m,k} (2\gamma) = \delta_{nk},
\end{equation}
for all $\gamma \geq \frac{N-3}{2}$. Similar arguments apply in the beta' case. Switching back to the original variable $\beta$, we arrive at the following result which is the dual of Proposition~\ref{prop:relations}.
\begin{proposition}
For every $n\in \{2,3,\ldots\}$, $k\in \{1,\ldots,n-1\}$ and every $\beta\geq -1$ we have
\begin{align}
\sum_{\substack{s=0,1,\ldots\\ n-s\geq k}}  (-1)^{s} \bJ_{n,n-s}(\beta)  \bI_{n-s,k} (2\beta+n-1) = 0.
%\sum_{\substack{s=0,1,\ldots\\ n-2s-1\geq k}} \bI_{n,n-2s-1}(2\beta+n-1) \bJ_{n-2s-1,k} \left(\beta + s + \frac 12\right) = \frac 12 \binom nk.
\label{eq:relation_I_J_2_dual}
\end{align}
Similarly, for every $n\in \{2,3,\ldots\}$, $k\in \{1,\ldots,n-1\}$ and for every $\beta>(n-1)/2$, we have
\begin{align}
\sum_{\substack{s=0,1,\ldots\\ n-s\geq k}} (-1)^{s} \tilde \bJ_{n,n-s}(\beta) \tilde \bI_{n-s,k} (2\beta-n+1) =  0.
\label{eq:relation_I_J_2_tilde_dual}
\end{align}
\end{proposition}

%\subsection{Arithmetic properties of $\bJ_{n,k}(\beta)$ and $\tilde \bJ_{n,k}(\beta)$}
\subsection{Arithmetic properties of expected internal-angle sums}
At the moment, we do not have a general formula for $\bJ_{n,k}(\beta)$ and $\tilde \bJ_{n,k}(\beta)$ which is ``nicer'' than what is given in Theorems~\ref{theo:formula_J_1},\ref{theo:formula_J_2},\ref{theo:formula_J_1_tilde},\ref{theo:formula_J_2_tilde}. Still, we can say something about the arithmetic properties of these quantities. First we state what we know about $\bI_{n,k}(\alpha)$.

\begin{theorem}\label{theo:I_arithm}
Let $\alpha\geq 0$ be integer, $n\in\N$ and $k\in \{1,\ldots,n\}$.
\begin{itemize}
\item[(a)]
If $\alpha$ is odd, then $\bI_{n,k}(\alpha)$ is rational.
\item[(b)]
If  $\alpha$ is even, then $\bI_{n,k}(\alpha)$ can be expressed in the form $r_0+r_2\pi^{-2} + r_4\pi^{-4} +\ldots + r_{n-k} \pi^{-(n-k)}$ (if $n-k$ is even) or $r_0+r_2\pi^{-2} + r_4\pi^{-4} +\ldots + r_{n-k-1} \pi^{-(n-k-1)}$ (if $n-k$ is odd), where the $r_i$'s are rational numbers.
\end{itemize}
\end{theorem}

Using the above theorem together with the results of Section~\ref{sec:relations}, we shall prove the following result on the $\bJ_{n,k}(\beta)$'s.
\begin{theorem}%[Arithmetic properties of $\bJ_{n,k}(\beta)$]
\label{theo:arithm_J}
Let $\beta\geq -1$ be integer or half-integer. Let also $n\in\N$ and $k\in \{1,\ldots,n\}$.
\begin{itemize}
\item[(a)] If $2\beta + n$ is even, then $\bJ_{n,k}(\beta)$ is a rational number.
%\item[(b)] If $2\beta + n$ is odd and $n-k$ is odd as well, then $\bJ_{n,k}(\beta)$ is a number of the form $q\pi^{-(n-k-1)}$ with some rational $q$.
\item[(b)] If $2\beta + n$ is odd, then $\bJ_{n,k}(\beta)$ can be expressed as  $q_0 + q_2 \pi^{-2} + q_4 \pi^{-4} + \ldots + q_{n-k} \pi^{-(n-k)}$ (if $n-k$ is even) or $q_0 + q_2 \pi^{-2} + q_4 \pi^{-4} + \ldots + q_{n-k} \pi^{-(n-k-1)}$ (if $n-k$ is odd),
     where the $q_{i}$'s are rational numbers.
%     Finally, $q_{n-k}$ vanishes if $k\in \{1,2\}$.
\end{itemize}
\end{theorem}

Symbolic computations we performed with the help of Mathematica~11 strongly suggest that in the case when $n-k$ is odd, Part~(b) can be strengthened as follows:
\begin{conjecture}\label{conj:J_arithm}
If both $2\beta + n$  and $n-k$ are odd, then $\bJ_{n,k}(\beta)$ is a number of the form $q\pi^{-(n-k-1)}$ with some rational $q$.
\end{conjecture}
%We shall prove Theorem~\ref{theo:arithm_J} using Theorems~\ref{eq:alg_J_1} and~\ref{eq:alg_J_2} in Section
Conjecture~\ref{conj:J_arithm} states that $\bJ_{n,k}(\beta)$ has sometimes much simpler form than the one suggested by Theorems~\ref{theo:formula_J_1} and~\ref{theo:formula_J_2}.  For example, when computing $\bJ_{7,2}(-1)$, we can use the formula
$$
\bJ_{7,2}(-1) = \frac 12 \binom 72 - \bI_{7,5}(4) \bJ_{5,2}(0) - \bI_{7,3}(4) \bJ_{3,2}(1)
$$
that follows from~\eqref{eq:alg_J_2}.
The involved values are given by $\bJ_{5,2}(0) = \frac{1692197}{282240 \pi ^2}$, $\bJ_{3,2}(1) = \frac 32$ and
$$
\bI_{7,5}(4) = 7-\frac{2144238917}{190270080 \pi ^2},
\quad
\quad
\bI_{7,3}(4) = 7 + \frac{1250163908136617}{30981823488000 \pi ^4}-\frac{1692197}{60480 \pi ^2},
$$
so that, a priori, we expect $\bJ_{7,2}(-1)$ to be a linear combination of $1, \pi^{-2}, \pi^{-4}$ over $\bQ$. A posteriori, it turns out that $\bJ_{7,2}(-1)=\frac{113537407}{16128000 \pi ^4}$ is a rational multiple of $\pi^{-4}$, while the remaining terms cancel. We were not able to explain this strange cancellation using Theorems~\ref{theo:formula_J_1} and~\ref{theo:formula_J_2}. It is therefore natural to conjecture that there is a ``nicer'' formula for $\bJ_{n,k}(\beta)$ than the ones given in these theorems.
 %formula for $\bJ_{n,k}(\beta)$ which explains Conjecture~\ref{conj:J_arithm}.

The results for the quantities $\tilde \bI_{n,k}(\alpha)$ and $\tilde \bJ_{n,k}(\beta)$ are analogous. We state them without proofs.
\begin{theorem}\label{theo:I_arithm_tilde}
Let $\alpha>0$ be integer, $n\in\N$ and $k\in \{1,\ldots,n\}$.
\begin{itemize}
\item[(a)]
If $\alpha$ is even, then $\tilde \bI_{n,k}(\alpha)$ is rational.
\item[(b)]
If  $\alpha$ is odd, then $\tilde \bI_{n,k}(\alpha)$ can be expressed in the form $r_0+r_2\pi^{-2} + r_4\pi^{-4} +\ldots + r_{n-k} \pi^{-(n-k)}$ (if $n-k$ is even) or $r_0+r_2\pi^{-2} + r_4\pi^{-4} +\ldots + r_{n-k-1} \pi^{-(n-k-1)}$ (if $n-k$ is odd), where the $r_i$'s are rational numbers.
\end{itemize}
\end{theorem}

\begin{theorem}%[Arithmetic properties of $\tilde \bJ_{n,k}(\beta)$]
\label{theo:arithm_J_tilde}
Let $n\in\N$ and  $k\in \{1,\ldots,n\}$. Let also $\beta > (n-1)/2$ be integer or half-integer.
\begin{itemize}
\item[(a)] If $2\beta - n$ is odd, then $\tilde \bJ_{n,k}(\beta)$ is a rational number.
%\item[(b)] If both $2\beta - n$ and $n-k$ are even, then $\tilde \bJ_{n,k}(\beta)$ is a number of the form $q\pi^{-(n-k)}$ with some rational $q$.
\item[(b)] If $2\beta - n$ is even, then $\tilde \bJ_{n,k}(\beta)$ can be expressed as
$q_0 + q_2 \pi^{-2} + q_4 \pi^{-4} + \ldots + q_{n-k-1} \pi^{-(n-k-1)}$ (if $n-k$ is odd) or
$q_0 + q_2 \pi^{-2} + q_4 \pi^{-4} + \ldots + q_{n-k} \pi^{-(n-k)}$ (if $n-k$ is even), where the $q_i$'s are rational numbers.
%a polynomial of $\pi^{-2}$ with rational coefficients:
%     $$
%     \tilde \bJ_{n,k}(\beta) = q_0 + q_2 \pi^{-2} + q_4 \pi^{-4} + \ldots + q_{n-k-1} %\pi^{-(n-k-1)},
%     $$
%     where $q_{0}, q_2,\ldots, q_{n-k}$ are rational numbers.
%     Finally, $q_{n-k-1}$ vanishes if $k\in \{1,2\}$.
\end{itemize}
\end{theorem}

In the case when $k$ is even, our symbolic computations suggest the following stronger version of Part~(b):
\begin{conjecture}\label{conj:J_tilde_arithm}
If both $2\beta - n$ and $k$ are even, then $\tilde \bJ_{n,k}(\beta)$ is a number of the form $q\pi^{-(n-k)}$ (if $n-k$ is even) or $q\pi^{-(n-k-1)}$ (if $n-k$ is odd) with some rational $q$.
\end{conjecture}

%%%%%
%%%Here I removed a section on affine invariance of expected internal angles

\section{Special cases and applications}\label{sec:applications}
In this section we present several special cases of the above results and their applications to some problems of stochastic geometry. The symbolic computations were performed using Mathematica~11. For the vector of the expected internal angles we use the  notation
$$
\bJ_{n,\bullet} (\beta) = (\bJ_{n,1}(\beta), \ldots, \bJ_{n,n}(\beta)).
$$

\subsection{Internal angles of random simplices: Uniform distribution on the sphere}
Let $X_1,\ldots,X_n$ be i.i.d.\ random points sampled uniformly from the unit sphere $\bS^{n-2} \subset \R^{n-1}$. Recall that the expected sum of internal angles of the simplex $[X_1,\ldots,X_n]$ at its $k$-vertex faces is denoted by $\bJ_{n,k}(-1)$. Clearly, we have the trivial results
\begin{equation}\label{eq:J_trivial}
\bJ_{1,\bullet}(-1) = (1),
\qquad
\bJ_{2,\bullet}(-1) = (1,1),
\qquad
\bJ_{3,\bullet}(-1) = \Bigg(\frac 12,\frac 32, 1\Bigg).
\end{equation}
The first two non-trivial cases, $n=3$ and $n=4$ (corresponding to simplices in dimensions $3$ and $4$), were  treated in~\cite{kabluchko_angles}:
\begin{equation}
\bJ_{4,\bullet}(-1) = \Bigg(\frac 18,\frac 98,2,1\Bigg),
\qquad
\bJ_{5,\bullet}(-1) = \Bigg(-\frac 16 + \frac{539}{288 \pi^2},\frac{539}{96 \pi^2},\frac 53 + \frac{539}{144 \pi^2},\frac 52,1\Bigg).
\end{equation}
The method used there did not allow for an extension to higher dimensions. Using Mathematica~11 and the algorithm described in Section~\ref{sec:relations} we recovered these results and, moreover, obtained the following
\begin{theorem}
We have
{\tiny{
\begin{flalign*}
%\bJ_{1,\bullet}(-1) &= (1),&\\
%\bJ_{2,\bullet}(-1) &= (1,2),&\\
%\bJ_{3,\bullet}(-1) &= \Bigg(\frac 12,\frac 32, 1\Bigg),&\\
%\bJ_{4,\bullet}(-1) &= \Bigg(\frac 18,\frac 98,2,1\Bigg),&\\
%\bJ_{5,\bullet}(-1) &= \Bigg(-\frac 16 + \frac{539}{288 \pi^2},\frac{539}{96 \pi^2},\frac 53 + \frac{539}{144 \pi^2},\frac 52,1\Bigg),&\\
\bJ_{6,\bullet}(-1) &= \Bigg(\frac{25411}{7340032},\frac{233445}{1048576},\frac{5155}{3584},\frac{23075}{7168},3,1\Bigg),&\\
\bJ_{7,\bullet}(-1) &= \Bigg(\frac{1}{6}+\frac{113537407}{48384000 \pi ^4}-\frac{2144238917}{1141620480 \pi ^2},\frac{113537407}{16128000 \pi ^4},-\frac{7}{6}+\frac{113537407}{24192000 \pi ^4}+\frac{2144238917}{114162048 \pi ^2},\frac{2144238917}{76108032 \pi ^2},&\\
&\phantom{=\Bigg(}  \frac{7}{2}+\frac{2144238917}{190270080 \pi ^2},\frac{7}{2},1\Bigg),&\\
\bJ_{8,\bullet}(-1) &= \Bigg(
\frac{76136856565967}{1454662679640670208},
\frac{29503701837953231}{1454662679640670208},
\frac{5899486844923}{16647293239296},
\frac{1146031403475}{584115552256},
\frac{418431615}{84672512},
\frac{1603846783}{254017536},
4, 1\Bigg),&\\
\bJ_{9,\bullet}(-1) &= \Bigg(-\frac{3}{10}-\frac{1581133359667623075371927}{218521780048552780800000 \pi ^4}+\frac{2819369438967901759}{1739761680384000000 \pi ^6}+\frac{3585828150520517221}{975094112225376000 \pi ^2},&\\
&\phantom{=\Bigg(} \frac{2819369438967901759}{579920560128000000 \pi ^6},\frac{1581133359667623075371927}{21852178004855278080000 \pi ^4}+\frac{2819369438967901759}{869880840192000000 \pi ^6}+2-\frac{25100797053643620547}{975094112225376000 \pi ^2},&\\
&\phantom{=\Bigg(} \frac{1581133359667623075371927}{14568118669903518720000 \pi ^4},-\frac{21}{5}+\frac{1581133359667623075371927}{36420296674758796800000 \pi ^4}+\frac{25100797053643620547}{325031370741792000 \pi ^2},&\\
&\phantom{=\Bigg(} \frac{25100797053643620547}{325031370741792000 \pi ^2}, 6+\frac{3585828150520517221}{162515685370896000 \pi ^2},\frac{9}{2},1\Bigg),&\\
\bJ_{10,\bullet}(-1) &= \Bigg(\frac{7142769685117513413611137831}{13319284084760520585863454122835968},\frac{15207860904181118336356297648935}{13319284084760520585863454122835968},&\\
&\phantom{=\Bigg(} \frac{9440668036340000013447895}{198472799133666166452518912},\frac{240195630998707566620445}{441541266148311827480576},&\\
&\phantom{=\Bigg(} \frac{65392213852270069737}{23659801379879256064},\frac{177147685252097540771}{23659801379879256064},&\\
&\phantom{=\Bigg(} \frac{8199101438535}{705117028352},\frac{29352612289095}{2820468113408},5,1\Bigg).&
\end{flalign*}
}}
\end{theorem}

\subsection{Internal angles of random simplices: Uniform distribution in the ball}
Let $X_1,\ldots,X_n$ be i.i.d.\ random points sampled uniformly from the unit ball $\bB^{n-1}$. %(or, more generally, from any $(n-1)$-dimensional ellipsoid of the form $\{x\in\R^{n-1}\colon \langle Qx, x\rangle\leq 1\}$, where $Q$ is a positive-definite $(n-1)\times (n-1)$-matrix).
The expected sum of internal angles of the simplex $[X_1,\ldots,X_n]$ at its $k$-vertex faces is $\bJ_{n,k}(0)$.
%(see Proposition~\ref{prop:elliptic} for the case of the ellipsoid).
The values of $\bJ_{n,k}(0)$ for $n=1,2,3$ are the same as in~\eqref{eq:J_trivial}.
For simplices with $n=4$ and $n=5$ vertices (corresponding to dimensions $d=3$ and $4$), the following results were obtained in~\cite{kabluchko_angles} by a method not extending to higher dimensions:
\begin{equation*}
\bJ_{4,\bullet}(0) = \Bigg(\frac {401}{2560},\frac {2961}{2560},2,1\Bigg),
\qquad
\bJ_{5,\bullet}(0) = \Bigg(-\frac 16 + \frac{1692197}{846720 \pi^2},\frac{1692197}{282240 \pi^2},\frac 53 + \frac{1692197}{423360 \pi^2},\frac 52,1\Bigg).
\end{equation*}
Using Mathematica~11 and the above algorithm we recovered these results and, moreover, obtained the following
\begin{theorem}
We have
{\tiny{
\begin{flalign*}
%\bJ_{1,\bullet}(0)&= (1),&\\
%\bJ_{2,\bullet}(0)&= (1,2),&\\
%\bJ_{3,\bullet}(0)&= \Bigg(\frac 12,\frac 32, 1\Bigg),&\\
\bJ_{6,\bullet}(0)&= \Bigg(\frac{112433094897}{17197049053184},
\frac{29573170815}{120259084288},
\frac{6929155}{4685824},
\frac{30358275}{9371648},
3,
1\Bigg),&\\
\bJ_{7,\bullet}(0)
&=
\Bigg(
\frac{1}{6}+\frac{36051577693123}{13519341158400 \pi ^4}-\frac{621038966291119}{325969178895360 \pi ^2}, \frac{36051577693123}{4506447052800 \pi ^4},&\\
&\phantom{=\Bigg(} -\frac{7}{6}+\frac{36051577693123}{6759670579200 \pi ^4}+\frac{621038966291119}{32596917889536 \pi ^2}, \frac{621038966291119}{21731278593024 \pi ^2},
\frac{7}{2}+\frac{621038966291119}{54328196482560 \pi ^2}, \frac{7}{2}, 1
\Bigg),&\\
\bJ_{8,\bullet}(0)
&=
\Bigg(
\frac{54854407266470750437}{407304109147506899681280}, \frac{1922620195704749849441}{81460821829501379936256}, \frac{1818739186251799}{4855443348258816}, \frac{6494630010305885}{3236962232172544},&\\
&\phantom{=\Bigg(} \frac{2403490929}{482344960}, \frac{9156320369}{1447034880}, 4, 1
\Bigg),&\\
\bJ_{9,\bullet}(0)
&=
\Bigg(
-\frac{3}{10}-\frac{3825746278401786849105853842941927}{513083615323402301904101376000000 \pi ^4}+\frac{834997968128824111294853689}{434049888937072472064000000 \pi^6}&\\
&\phantom{=\Bigg(} \quad \quad +\frac{25695566187355249503645020401}{6950795362764910977640320000 \pi ^2},
\frac{834997968128824111294853689}{144683296312357490688000000 \pi ^6},&\\
&\phantom{=\Bigg(} \frac{3825746278401786849105853842941927}{51308361532340230190410137600000 \pi ^4}+\frac{834997968128824111294853689}{217024944468536236032000000 \pi ^6}+2-\frac{25695566187355249503645020401}{992970766109272996805760000 \pi ^2},&\\
&\phantom{=\Bigg(} \frac{3825746278401786849105853842941927}{34205574354893486793606758400000 \pi ^4},&\\
&\phantom{=\Bigg(} -\frac{21}{5}+\frac{3825746278401786849105853842941927}{85513935887233716984016896000000 \pi ^4}+\frac{25695566187355249503645020401}{330990255369757665601920000 \pi ^2},&\\
&\phantom{=\Bigg(} \frac{25695566187355249503645020401}{330990255369757665601920000 \pi ^2},6+\frac{25695566187355249503645020401}{1158465893794151829606720000 \pi ^2},\frac{9}{2},1
\Bigg)&,\\
\bJ_{10,\bullet}(0)
&=
\Bigg(\frac{16173937433865922950599394579005791588389155}{9204102262874833628227344732391414668379518140416},&\\
&\phantom{=\Bigg(} \frac{12688011280876667528205329700413092651546251555}{9204102262874833628227344732391414668379518140416},&\\
&\phantom{=\Bigg(} \frac{32929953220484140728052018125551175}{640848401352029148689993712621584384}, \frac{210765193340397846616524118474155}{373323101767213558740779093983232},&\\
&\phantom{=\Bigg(} \frac{371193086109705273947602629}{131859245100259540744536064},\frac{2253773101928857034270262735}{298418291542692644842897408},&\\
&\phantom{=\Bigg(} \frac{15529150935155595}{1330783805505536},\frac{55452665100321675}{5323135222022144},5,1\Bigg).&
\end{flalign*}
}}
\end{theorem}

\subsection{Typical Poisson-Voronoi cells}
Let $P_1,P_2,\ldots$ be the points of a Poisson point process on $\R^d$ with constant intensity $1$. %The Poisson-Voronoi tessellation is the decomposition of $\R^d$ into cells
%$$
%C(P_i) := \{x\in \R^d\colon \|x-P_i\| \leq \|x-P_j\| \text{ for all $j\in\N$}\}, \qquad i\in\N.
%$$
The \emph{typical Poisson-Voronoi cell} is a random polytope which, for our purposes, can be defined as follows:
$$
\mathcal V_d := \{x\in \R^d\colon \|x\| \leq \|x-P_j\| \text{ for all $j\in\N$}\}.
$$
The typical Poisson-Voronoi cell is one of the classical objects of stochastic geometry; see~\cite{SW08,moller,moller_book,calka_tess_rev,calka_tess_asympt_rev,hug_rev} for reviews and the works of  Meijering~\cite{meijering}, Gilbert~\cite{gilbert} and Miles~\cite{miles_synopsis} for important early contributions.
We shall be interested in the expected $f$-vector of $\mathcal V_d$ denoted by
$$
\E \mathbf f(\mathcal V_{d}) = (\E f_0(\mathcal V_{d}),\E f_1(\mathcal V_{d}),\ldots,\E f_{d-1}(\mathcal V_{d})),
$$
where $f_k(\mathcal V_d)$ is the number of $k$-dimensional faces of $\mathcal V_d$.
To the best of our knowledge, explicit formulae for the complete vector $\E \mathbf f(\mathcal V_{d})$ have been known only in dimensions $d=2$ and $3$:
\begin{equation}\label{eq:Poisson_Voronoi_d=2,3}
\E \mathbf f (\mathcal V_{2}) = (6, 6),
\qquad
\E \mathbf f (\mathcal V_{3}) = \Bigg(\frac{96 \pi ^2}{35},\frac{144 \pi ^2}{35},2+\frac{48 \pi ^2}{35}\Bigg),
\end{equation}
see~\cite[Theorem~10.2.5]{SW08} or~\cite[Equation~(7.13)]{moller}.
The following formula can be found in the works of Miles~\cite[Equation~(75)]{miles_synopsis} and M\o ller~\cite[Theorem~7.2]{moller}:
\begin{equation}\label{eq:moller}
\E f_0(\mathcal V_d) = \frac {2^{d+1}\pi^{(d-1)/2}}{d^2} \frac{\Gamma(\frac{d^2+1}{2})}{\Gamma(\frac {d^2}{2})}
\left(\frac{\Gamma(\frac {d+2}2)}{\Gamma(\frac{d+1}{2})}\right)^d.
\end{equation}
In fact, there is a more general formula~\cite[Theorem~7.2]{moller} for the expected $s$-content of all $s$-faces of a typical $t$-face in a $d$-dimensional tessellation, but it is only the case $s=0$, $t=d$ for which this results yields a formula for some entry of the expected $f$-vector of $\mathcal V_d$.

For arbitrary $d\in\N$ and for all $k\in \{0,\ldots,d-1\}$, it has been shown in~\cite{beta_polytopes} (see Theorem~1.21 and its proof there, with $\alpha = d$) that
\begin{equation}\label{eq:poisson_voronoi_f_k}
\E f_k(\mathcal V_{d})
= 2\sum_{\substack{m\in \{d-k,\ldots,d\}\\ m\equiv d \Mod{2}}} \tilde  \bI_{\infty,m}(d) \tilde{\bJ}_{m,d-k}\left({m-1+d\over 2}\right),
\end{equation}
where
\begin{equation}\label{eq:poisson_voronoi_f_k_addition}
\tilde \bI_{\infty,m}(d)
:=
\lim_{n\to\infty} \tilde \bI_{n,m}(d)
=
\frac{\tilde c_{1, \frac{dm+1}{2}}}{\tilde c_{1,\frac{d+1}{2}}^m} \cdot \frac{d^{m-1}}{m}
=
{\Gamma({\frac{md + 1} 2})\over\Gamma({md\over 2})}\left({\Gamma({d\over 2})\over\Gamma({d+1\over 2})}\right)^{m}{(\sqrt{\pi}d)^{m-1}\over m}.
\end{equation}
Taking $k=0$, we recover~\eqref{eq:moller}. Formula~\eqref{eq:poisson_voronoi_f_k}, together with the algorithm for computation of $\tilde \bJ_{n,k}(\beta)$, allows us to compute $\E f_k(\mathcal V_{d})$ in finitely many steps. Using Mathematica~11, we have done this in dimensions $d\in \{2,\ldots,10\}$. As a result, we recovered~\eqref{eq:Poisson_Voronoi_d=2,3} and, moreover, obtained the following
%In the following, we give explicit formulae for the expected $f$-vector of the typical cell  of the $d$-dimensional Poisson--Voronoi tessellation, for .
\begin{theorem}\label{theo:poisson_voronoi}
The expected $f$-vector of the typical Poisson-Voronoi cell is given by
{\tiny
\begin{flalign*}
%\E \mathbf f (\mathcal V_{2}) &= (6, 6),&\\
%\E \mathbf f (\mathcal V_{3}) &= \Bigg(\frac{96 \pi ^2}{35},\frac{144 \pi ^2}{35},2+\frac{48 \pi ^2}{35}\Bigg),&\\
\E \mathbf f (\mathcal V_{4}) &= \Bigg(\frac{1430}{9},\frac{2860}{9},\frac{590}{3},\frac{340}{9}\Bigg),&\\
\E \mathbf f (\mathcal  V_{5}) &= \Bigg(\frac{7776000 \pi ^4}{676039},\frac{19440000 \pi ^4}{676039},\frac{2716500 \pi ^2}{49049}+\frac{12960000 \pi ^4}{676039},\frac{4074750 \pi ^2}{49049},2+\frac{1358250 \pi ^2}{49049}-\frac{1296000 \pi ^4}{676039}\Bigg),&\\
\E \mathbf f (\mathcal  V_{6}) &= \Bigg(\frac{90751353}{10000},\frac{272254059}{10000},\frac{120613311}{4000},\frac{14930979}{1000},\frac{62611437}{20000},\frac{4053}{20}\Bigg),&\\
\E \mathbf f (\mathcal  V_{7}) &=\Bigg(\frac{27536588800000 \pi ^6}{322476036831},\frac{96378060800000 \pi ^6}{322476036831},\frac{145800103122713984000 \pi ^4}{139352342399730603}+\frac{96378060800000 \pi ^6}{322476036831},&\\
&\phantom{=\Bigg(} \frac{364500257806784960000 \pi ^4}{139352342399730603},
\frac{1088840823954800 \pi ^2}{1430074210851}+\frac{729000515613569920000 \pi ^4}{418057027199191809}-\frac{96378060800000 \pi ^6}{967428110493},&\\
&\phantom{=\Bigg(} \frac{544420411977400 \pi ^2}{476691403617},2+\frac{544420411977400 \pi ^2}{1430074210851}-\frac{72900051561356992000 \pi ^4}{418057027199191809}+\frac{13768294400000 \pi ^6}{967428110493}\Bigg) ,&\\
\E \mathbf f (\mathcal V_{8}) &= \Bigg(\frac{37400492672297766}{45956640625},\frac{149601970689191064}{45956640625},\frac{6850391092580412}{1313046875},\frac{27954881044110648}{6565234375},\frac{17044839181035378}{9191328125},&\\
&\phantom{=\Bigg(} \frac{18843745433119128}{45956640625},\frac{5212716470964}{133984375},\frac{4422456}{4375}\Bigg),&\\
\E \mathbf f (\mathcal V_{9}) &= \Bigg(\frac{100837904362675200000000 \pi ^8}{109701233401363445369},\frac{453770569632038400000000 \pi ^8}{109701233401363445369},&\\
&\phantom{=\Bigg(}\frac{2852955835216853216138612837266320000 \pi ^6}{134952926502386519274273464063983}+\frac{605027426176051200000000 \pi ^8}{109701233401363445369},&\\
&\phantom{=\Bigg(}\frac{9985345423258986256485144930432120000 \pi ^6}{134952926502386519274273464063983},&\\
&\phantom{=\Bigg(} \frac{16352535012213243758810504565072375 \pi ^4}{326981148443273530305985029716}+\frac{9985345423258986256485144930432120000 \pi ^6}{134952926502386519274273464063983}&\\
&\phantom{=\Bigg(} \quad \quad -\frac{423519198323235840000000 \pi ^8}{109701233401363445369},&\\
&\phantom{=\Bigg(}\frac{81762675061066218794052522825361875 \pi ^4}{653962296886547060611970059432},&\\
&\phantom{=\Bigg(}\frac{19758536784497995373925 \pi ^2}{2249321131934361056}+\frac{27254225020355406264684174275120625 \pi ^4}{326981148443273530305985029716}&\\
&\phantom{=\Bigg(} \quad\quad -\frac{3328448474419662085495048310144040000 \pi ^6}{134952926502386519274273464063983}+\frac{201675808725350400000000 \pi ^8}{109701233401363445369},&\\
&\phantom{=\Bigg(}\frac{59275610353493986121775 \pi ^2}{4498642263868722112},&\\
&\phantom{=\Bigg(}2+\frac{19758536784497995373925 \pi ^2}{4498642263868722112}-\frac{5450845004071081252936834855024125 \pi ^4}{653962296886547060611970059432}&\\
&\phantom{=\Bigg(}\quad \quad  +\frac{475492639202808869356435472877720000 \pi ^6}{134952926502386519274273464063983}-\frac{30251371308802560000000 \pi ^8}{109701233401363445369}\Bigg) ,&\\
\E \mathbf f (\mathcal V_{10})&= \Bigg(\frac{155696519360438569961130397}{1556433053837891712},\frac{778482596802192849805651985}{1556433053837891712}, \frac{363290492786125188681583835}{345874011963975936},&\\
&\phantom{=\Bigg(}
\frac{4865451274315354941235930}{4053211077702843},\frac{89845553163656455297282315}{111173789559849408},\frac{23998744131568764316595507}{74115859706566272},&\\
&\phantom{=\Bigg(}\frac{32972345885500895805463345}{444695158239397632},\frac{377982052291467600549815}{43234251495496992},\frac{5889025850448565}{13894111602},&\\
&\phantom{=\Bigg(}\frac{402700265}{83349}\Bigg).&
\end{flalign*}
}
\end{theorem}

Combining~\eqref{eq:poisson_voronoi_f_k} with Theorem~\ref{theo:arithm_J_tilde}, we can say something about the arithmetic structure of   $\E f_k(\mathcal V_d)$ for arbitrary dimension $d$.
\begin{theorem}\label{theo:poisson_voronoi_arithm}
Let $d\in\N$ and $k\in \{0,\ldots,d-1\}$.
\begin{itemize}
\item[(a)] If $d$ is even, then $\E f_k(\mathcal V_{d})$ is a rational number.
\item[(b)] If $d$ is odd, then $\E f_k(\mathcal V_{d})$ can be expressed as  $q_{d-1} \pi^{d-1} + q_{d-3} \pi^{d-3} + \ldots +q_{d-k-1}\pi^{d-k-1}$, (if $k$ is even) or $q_{d-1} \pi^{d-1} + q_{d-3} \pi^{d-3} + \ldots +q_{d-k}\pi^{d-k}$ (if $k$ is odd), where the coefficients $q_{i}$ are rational.
%     Finally, $q_{d-3}$ vanishes if $k = 2$.
\end{itemize}
\end{theorem}

%\begin{proof}[Proof of Theorem~\ref{theo:poisson_voronoi_arithm}]
%\subsection{Proof of Theorem~\ref{theo:poisson_voronoi_arithm}}
%Recall from~\eqref{eq:poisson_voronoi_f_k} and~\eqref{eq:poisson_voronoi_f_k_addition} that
%\begin{equation}\label{eq:E_f_k_poi_voronoi_repetition}
%\E f_k(\mathcal V_{d})
% 2\sum_{\substack{m\in \{d-k,\ldots,d\}\\ m\equiv d \Mod{2}}}
%\tilde \bI_{\infty,m}(d)\tilde{\bJ}_{m,d-k}\left({m-1+d\over 2}\right)
%\end{equation}
%with
%$$
%\tilde \bI_{\infty,m}(d) = {\Gamma({md+1\over 2})\over\Gamma({md\over 2})}\left({\Gamma({d\over 2})\over\Gamma({d+1\over %2})}\right)^{m}{(\sqrt{\pi}d)^{m-1}\over m}.
%$$

%\vspace*{2mm}
%\noindent
%\textit{}
\begin{proof}[Proof of Part (a).] Let $d$ be even. Recall that $\Gamma(x)$ is integer if $x>0$ is integer, and is a rational multiple of $\sqrt \pi$ if $x>0$ is half-integer. It follows from~\eqref{eq:poisson_voronoi_f_k_addition} that $\tilde \bI_{\infty,m}(d)$ is rational. Also, by Theorem~\ref{theo:arithm_J_tilde}, Part~(a),   $\tilde{\bJ}_{m,d-k}({m-1+d\over 2})$ is rational. It follows from~\eqref{eq:poisson_voronoi_f_k} that $\E f_k(\mathcal V_{d})$ is rational.
\end{proof}

%\vspace*{2mm}
%\noindent
%\textit{Proof of Part (b).}
\begin{proof}[Proof of Part (b).]
Let now $d$ be odd. The summation in~\eqref{eq:poisson_voronoi_f_k} is over odd values of $m$. For any such value, $\tilde \bI_{\infty,m}(d)$ is a rational multiple of $\pi^{m-1}$.  On the other hand, by Theorem~\ref{theo:arithm_J_tilde}, Part~(b),  $\tilde{\bJ}_{m,d-k}({m-1+d\over 2})$ can be written as a $\bQ$-linear combination of $\pi^{-j}$, where $j$ is even and satisfies $0\leq j\leq m-d+k$.
It follows that $\tilde \bI_{\infty,m}(d)\tilde{\bJ}_{m,d-k}({m-1+d\over 2})$ is a $\bQ$-linear combination of $\pi^{\ell}$, $d-k-1\leq \ell \leq m-1$, with $\ell\equiv m-1 \equiv d-1 \Mod{2}$. The claim follows.
\end{proof}
%\hfill $\Box$

%\vspace*{2mm}
In fact, a closer look at the values collected in Theorem~\ref{theo:poisson_voronoi} suggests the following conjecture which is a consequence of Conjecture~\ref{conj:J_tilde_arithm}. %and the fact that $\Gamma(m+\frac 12)$ is a rational multiple of $\Gamma(\frac 12) = \sqrt \pi$ for all $m\in\N$:
\begin{conjecture}
If both $d$ and $k$ are odd, then $\E f_k(\mathcal V_{d})$ is a number of the form $q\pi^{d-k}$ with some rational $q$.
\end{conjecture}

%Let us also mention that there is one more random polytope, the zero cell of the Poisson hyperplane tessellation, whose expected $f$-vector can be expressed through the quantities $\tilde \bJ_{n,m}(n/2)$. We refer to~\cite{kabluchko_poisson_zero} for details.

\subsection{Random polytopes approximating smooth convex bodies}\label{subsec:rand_poly_appr_conv_body}
Let $U_1,U_2,\ldots$ be independent random points distributed uniformly in the $d$-dimensional convex body $K$. Denote the convex hull of $n$ such points by $K_{n,d}= [U_1,\ldots,U_n]$. Asymptotic properties of $K_{n,d}$, as $n\to\infty$, have been very much studied starting with the work of R\'enyi and Sulanke~\cite{renyi_sulanke1,renyi_sulanke2} (see, for example, \cite{schneider_polytopes,hug_rev}) and  we shall not attempt to review the vast literature on this topic.
In particular, regarding the $f$-vector of $K_{n,d}$, this development culminated in the work of Reitzner~\cite{ReitznerCombinatorialStructure} who proved the following result~\cite[p.~181]{ReitznerCombinatorialStructure}. If the boundary of $K$ is of differentiability class $\mathcal C^2$ and the Gaussian curvature $\kappa(x)>0$ is positive at every boundary point $x\in \partial K$, then
\begin{equation}\label{eq:ReitznerExpectation}
\lim_{n\to\infty}
n^{-{d-1\over d+1}}\E f_k(K_{n,d})
=
c_{d,k}\,\Omega(K)/\Vol_d(K)^{\frac{d-1}{d+1}},
\end{equation}
for every $k\in\{0,1,\ldots,d-1\}$, where $\Omega(K):=\int_{\partial K}\kappa(x)^{1\over d+1} \dint x$ is the so-called affine surface area of $K$, and $c_{d,0}, \ldots, c_{d,d-1}$ are certain strictly positive constants not depending on $K$. In~\cite{ReitznerCombinatorialStructure}, Equation~\eqref{eq:ReitznerExpectation} is stated without the term involving $\Vol_d(K)$, for which it is necessary to assume that $K$ has unit volume. The general case follows from the following scaling property of the affine surface area:
$$
\Omega(r K) = r^{\frac{d(d-1)}{d+1}}\Omega (K), \qquad r>0,
$$
see, e.g., Theorem~3.6 in~\cite{hug_affine} and take $p=1$ there.

As Reitzner~\cite[p.~181]{ReitznerCombinatorialStructure} writes, ``It would be of interest to determine the vector $\mathbf c_d= (c_{d,0},\ldots, c_{d,d-1})$; but we have not succeeded in getting an explicit expression''.
Our aim is to provide explicit expressions for $\mathbf c_d$ for all $d\leq 10$. In the following, it will be convenient to take $K := \bB^d$ (which is possible since $\mathbf c_d$ does not depend on $K$) and use the notation
\begin{equation}\label{eq:def_C_d_k}
C_{d,k}
:=
\lim_{n\to\infty}
n^{-{d-1\over d+1}}\E f_k(P_{n,d}^0)
=
c_{d,k}\, \Omega(\bB^d)/\Vol_d(\bB^d)^{\frac{d-1}{d+1}}
=
\frac{d\cdot  \pi^{\frac d {d+1}}}{\Gamma(1+\frac d2)^{\frac 2 {d+1}}} \cdot c_{d,k},
\end{equation}
for all $k\in\{0,1,\ldots,d-1\}$.
Here, we recall that $P_{n,d}^0 = [X_1,\ldots,X_n]$ is the convex hull of $n$ i.i.d.\ random points $X_1,\ldots,X_n$ distributed uniformly in the ball $\bB^d$. Note also that the affine surface area of the unit ball coincides with its usual surface area: $\Omega(\bB^d)=2\pi^{d/2}/\Gamma({d\over 2})$. For $d=2$, the value of $C_{2,0}=C_{2,1}$ has been identified by R\'enyi and Sulanke~\cite[Satz~3]{renyi_sulanke1} who proved that
$$
C_{2,0} = C_{2,1} =  \lim_{n\to\infty} n^{-1/3} \E f_1(P_{n,2}^0) = \lim_{n\to\infty} n^{-1/3} \E f_0(P_{n,2}^0) = 2 \Gamma(5/3) \pi^{2/3} \sqrt[3]{2/3} .
$$
If $d\in\N$ is arbitrary and $k=d-1$, Affentranger~\cite{affentranger} (see his Corollary~1 on p.~366, the formula for $c_3$ on p.~378, and take $q=0$) proved that
\begin{equation}\label{eq:Limitf_kBall_affentranger}
C_{d,d-1}
%= \lim_{n\to\infty} n^{-\frac{d-1}{d+1}} \E f_{d-1}(P_{n,d}^0)
=
{2\pi^{d(d-1)\over 2(d+1)}\over (d+1)!}
{\Gamma(1+{d^2\over 2})\Gamma({d^2+1\over d+1})\over\Gamma({d^2+1\over 2})}\left({(d+1)\Gamma({d+1\over 2})\over \Gamma(1+{d\over 2})}\right)^{d^2+1\over d+1}.
\end{equation}
Note also that an exact formula for the number of facets of a convex hull of $N$ i.i.d.\ points sampled uniformly from the ball $\bB^d$ has been obtained by Buchta and M\"uller~\cite{buchta_mueller} (see their Theorem~3 on page~760), but it requires some work to analyze its asymptotic behavior as $N\to\infty$.
%mentions:
In~\cite[Remark~1.9]{beta_polytopes}, it has been shown that for all $d\in\N$ and $k\in\{0,\ldots,d-1\}$, we have
\begin{equation}\label{eq:Limitf_kBall}
C_{d,k}
%= \lim_{n\to\infty} n^{-\frac{d-1}{d+1}}\E f_k(P_{n,d}^0)
= {2\pi^{d(d-1)\over 2(d+1)}\over (d+1)!}
{\Gamma(1+{d^2\over 2})\Gamma({d^2+1\over d+1})\over\Gamma({d^2+1\over 2})}\left({(d+1)\Gamma({d+1\over 2})\over \Gamma(1+{d\over 2})}\right)^{d^2+1\over d+1} \bJ_{d,k+1}(1/2).
\end{equation}
In the special case $k=d-1$, equation~\eqref{eq:Limitf_kBall} reduces to~\eqref{eq:Limitf_kBall_affentranger} since $\bJ_{d,d}(1/2) = 1$. Hug~\cite[Corollary~7.1 and p.~209]{hug_rev} gave a formula for $c_{d,0}$ (and, hence, for $C_{d,0}$) which is equivalent to the formula for $\bJ_{d,1}(1/2)$ which will be stated in Theorem~\ref{theo:beta_+1/2} below.
Combining~\eqref{eq:Limitf_kBall} with the above algorithm for computing $\bJ_{d,k+1}(1/2)$, we obtain the following explicit formulae for Reitzner's constants in dimensions $d\leq 10$.

\begin{theorem}
The vectors $\mathbf C_d := (C_{d,0},\ldots, C_{d,d-1})$ are explicitly given by
{\tiny
\begin{flalign*}
\mathbf C_1 &= 2\cdot (1),&\\
\mathbf C_2 &= 2 \sqrt[3]{\frac{2}{3}} \pi ^{2/3} \Gamma \left(\frac{5}{3}\right) \times
(1,1),&\\
\mathbf C_3 &= \frac{35 \sqrt{\pi/3}}{4} \times \Bigg(\frac{1}{2},\frac{3}{2},1\Bigg),&\\
\mathbf C_4 &= \frac{20}{143} 2^{4/5} 15^{2/5} \pi ^{12/5} \Gamma \left(\frac{17}{5}\right) \times \Bigg(\frac{26741}{16800 \pi ^2},1+\frac{26741}{16800 \pi ^2},2,1\Bigg),&\\
\mathbf C_5 &= \frac{676039 \cdot \Gamma \left(\frac{13}{3}\right)}{18000 \sqrt[3]{10}} \times \Bigg(\frac{2000}{52003},\frac{64003}{104006},\frac{108006}{52003},\frac{5}{2},1\Bigg) ,&\\
\mathbf C_6 &= \frac{4390400\cdot 2^{6/7} 35^{2/7} \pi ^{30/7} \Gamma \left(\frac{37}{7}\right)}{116680311} \times
\Bigg(\frac{1758847651}{2458624000 \pi ^4},-\frac{1}{2}+\frac{1758847651}{2458624000 \pi ^4}+\frac{108130927981}{14717390688 \pi ^2},\frac{108130927981}{7358695344 \pi ^2},&\\
&\phantom{=\Bigg(}
\frac{5}{2}+\frac{108130927981}{14717390688 \pi ^2},3,1\Bigg),&\\
\mathbf C_7 &= \frac{35830670759 \cdot \Gamma \left(\frac{25}{4}\right)}{420175000 \sqrt[4]{35}} \times \Bigg(\frac{52521875}{44479453356},\frac{1260026621}{14826484452},\frac{708362065}{855374103},\frac{115870255}{39856141},\frac{371689191}{79712282},\frac{7}{2},1\Bigg),&\\
%&\phantom{=\Bigg(}
\mathbf C_8 &=
\frac{15752961000000\cdot 6^{4/9} 35^{2/9} \pi ^{56/9} \Gamma \left(\frac{65}{9}\right)}{2077805148460987} \times
\Bigg(\frac{90856752400884977}{571643448768000000 \pi ^6},&\\
&\phantom{=\Bigg(}
\frac{2}{3}+\frac{3883880966311229933975003293}{209349006975455882895360000 \pi ^4}+\frac{90856752400884977}{571643448768000000 \pi ^6}-\frac{486245776939428578826199}{59171148465116379120000 \pi ^2}
,&\\
&\phantom{=\Bigg(} \frac{3883880966311229933975003293}{104674503487727941447680000 \pi ^4},-\frac{7}{3}+\frac{3883880966311229933975003293}{209349006975455882895360000 \pi ^4}+\frac{486245776939428578826199}{11834229693023275824000 \pi ^2},&\\
&\phantom{=\Bigg(}
\frac{486245776939428578826199}{9861858077519396520000 \pi ^2},\frac{14}{3}+\frac{486245776939428578826199}{29585574232558189560000 \pi ^2},4,1\Bigg),&\\
\mathbf C_9 &=
\frac{109701233401363445369 \cdot \Gamma \left(\frac{41}{5}\right)}{726032911411261440\ 3^{2/5} \sqrt[5]{14}}\times
\Bigg(\frac{12004512424128}{581660834577748915},
\frac{3683565096070608}{581660834577748915},
\frac{17538430231527552}{116332166915549783},&\\
&\phantom{=\Bigg(}
\frac{570366050377039}{491890769198942},
\frac{1019018617306221}{245945384599471},
\frac{1080810073}{137168095},
\frac{1131811448}{137168095},\frac{9}{2},1\Bigg),&\\
\mathbf C_{10} &=
\frac{434735988912345551929344\cdot 2^{6/11} 3^{4/11} 77^{2/11} \pi ^{90/11} \Gamma \left(\frac{101}{11}\right)}{353855725819178568093478175} \times \Bigg(\frac{549837358580569775037558395}{24790385031737592753218912256 \pi ^8},&\\
&\phantom{=\Bigg(}
-\frac{3}{2}
-\frac{301974317327871030169614455148390753674792595873047}{6565687677840932855885309667960898754371584000000 \pi ^4}
+\frac{296364869518522313138595119776890880847603113}{11688440195468832553173084766502230425600000 \pi ^6}&\\
&\phantom{=\Bigg(}
\quad \quad +\frac{549837358580569775037558395}{24790385031737592753218912256 \pi ^8}+\frac{37401610118391599618484796905719320020269}{1946114861053154102938714818796281216000 \pi ^2},&\\
&\phantom{=\Bigg(}\frac{296364869518522313138595119776890880847603113}{5844220097734416276586542383251115212800000 \pi ^6},&\\
&\phantom{=\Bigg(}
\frac{301974317327871030169614455148390753674792595873047}{1313137535568186571177061933592179750874316800000 \pi ^4}+\frac{296364869518522313138595119776890880847603113}{11688440195468832553173084766502230425600000 \pi ^6}&\\
&\phantom{=\Bigg(} \quad \quad +5-\frac{37401610118391599618484796905719320020269}{556032817443758315125347091084651776000 \pi ^2},&\\
&\phantom{=\Bigg(}\frac{301974317327871030169614455148390753674792595873047}{1094281279640155475980884944660149792395264000000 \pi ^4},&\\
&\phantom{=\Bigg(}\frac{301974317327871030169614455148390753674792595873047}{3282843838920466427942654833980449377185792000000 \pi ^4}-7+\frac{37401610118391599618484796905719320020269}{278016408721879157562673545542325888000 \pi ^2},&\\
&\phantom{=\Bigg(}\frac{37401610118391599618484796905719320020269}{324352476842192350489785803132713536000 \pi ^2},\frac{15}{2}+\frac{37401610118391599618484796905719320020269}{1297409907368769401959143212530854144000 \pi ^2},5,1\Bigg).&
\end{flalign*}
}
\end{theorem}

\subsection{Random polytopes with vertices on the sphere}
Similarly, one can consider random polytopes approximating a convex body $K$ and having vertices on the boundary of $K$. Here, we restrict ourselves to the case $K=\bB^d$, so that we are interested in the random polytope $P_{n,d}^{-1}$ defined  as the convex hull of $n$ points $X_1,\ldots,X_n$ chosen uniformly at random on the unit sphere $\bS^{d-1}$, $d\geq 2$.  In~\cite[Remark~1.9]{beta_polytopes}, it has been shown that
\begin{equation}\label{eq:reitzner_const_sphere}
C_{d,k}^*:=\lim_{n\to\infty} \frac 1n  \E f_k(P_{n,d}^{-1})
={2^d\pi^{{d\over 2}-1}\over d(d-1)^2}{\Gamma(1+{d(d-2)\over 2})\over\Gamma({(d-1)^2\over 2})}\left({\Gamma({d+1\over 2})\over\Gamma({d\over 2})}\right)^{d-1} \bJ_{d,k+1}(-1/2),
\end{equation}
for all $k\in \{0,\ldots,d-1\}$.
In the special case $k=d-1$, it was previously shown by Affentranger~\cite{affentranger} (see his Corollary~1 on p.~366 and the formula for $c_3$ on p.~378, this time with $q=-1$) and Buchta, M\"uller, Tichy~\cite{buchta_mueller_tichy} (see their formula for $\bar F_n^{(d)}$ on p.~231) that
\begin{equation}\label{eq:C_star_d_d-1}
C_{d,d-1}^* = \frac{2^{d-1}} {d} \binom {d-1}{\frac 12 (d-1)}^{-(d-1)} \binom{(d-1)^2}{\frac12 (d-1)^2}
=
{2^d\pi^{{d\over 2}-1}\over d(d-1)^2}{\Gamma(1+{d(d-2)\over 2})\over\Gamma({(d-1)^2\over 2})}\left({\Gamma({d+1\over 2})\over\Gamma({d\over 2})}\right)^{d-1},
\end{equation}
where the equality of the expressions on the right-hand side follows from the duplication formula for the Gamma function.
This formula for $C_{d,d-1}^*$ is a special case of~\eqref{eq:reitzner_const_sphere} since $\bJ_{d,d}(-1/2) = 1$.
%Since $P_{n,d}^{-1}$ has $n$ vertices, we trivially have $C_{d,0}^*=1$. This observation, combined with~\eqref{eq:reitzner_const_sphere}, yields a formula for $\bJ_{d,1}(-1/2)$; see Theorem~\ref{theo:beta_-1/2} below.
Using~\eqref{eq:reitzner_const_sphere} together with the algorithm for computing $\bJ_{d,k+1}(-1/2)$, we obtain the following

\begin{theorem}
The vectors $\mathbf C_d^* := (C_{d,0}^*,\ldots, C_{d,d-1}^*)$ are explicitly given by
{\tiny
\begin{flalign*}
\mathbf C_2^* &=  (1,1),&\\
\mathbf C_3^* &=  (1,3,2),&\\
\mathbf C_4^* &= \Bigg(1,1+\frac{24 \pi ^2}{35},\frac{48 \pi ^2}{35},\frac{24 \pi ^2}{35}\Bigg),&\\
\mathbf C_5^* &=  \Bigg(1,\frac{170}{9},\frac{590}{9},\frac{715}{9},\frac{286}{9}\Bigg) ,&\\
\mathbf C_6^* &=  \Bigg(1,1+\frac{679125 \pi ^2}{49049}-\frac{648000 \pi ^4}{676039},\frac{1358250 \pi ^2}{49049},\frac{679125 \pi ^2}{49049}+\frac{3240000 \pi ^4}{676039},\frac{3888000 \pi ^4}{676039},\frac{1296000 \pi ^4}{676039}\Bigg) ,&\\
\mathbf C_7^* &=  \Bigg(1,\frac{4053}{40},\frac{20870479}{20000},\frac{14930979}{4000},\frac{120613311}{20000},\frac{90751353}{20000},\frac{12964479}{10000}\Bigg) ,&\\
\mathbf C_8^* &=  \Bigg(1,1+\frac{272210205988700 \pi ^2}{1430074210851}-\frac{36450025780678496000 \pi ^4}{418057027199191809}+\frac{6884147200000 \pi ^6}{967428110493},\frac{544420411977400 \pi ^2}{1430074210851},&\\
&\phantom{=\Bigg(}
\frac{272210205988700 \pi ^2}{1430074210851}+\frac{182250128903392480000 \pi ^4}{418057027199191809}-\frac{24094515200000 \pi ^6}{967428110493},
\frac{72900051561356992000 \pi ^4}{139352342399730603},&\\
&\phantom{=\Bigg(}
\frac{72900051561356992000 \pi ^4}{418057027199191809}+\frac{48189030400000 \pi ^6}{967428110493},\frac{13768294400000 \pi ^6}{322476036831},\frac{3442073600000 \pi ^6}{322476036831}\Bigg) ,&\\
\mathbf C_9^* &=  \Bigg(1,\frac{2211228}{4375},\frac{1737572156988}{133984375},\frac{4710936358279782}{45956640625},\frac{17044839181035378}{45956640625},\frac{4659146840685108}{6565234375},\frac{6850391092580412}{9191328125},&\\
&\phantom{=\Bigg(}\frac{18700246336148883}{45956640625},\frac{4155610296921974}{45956640625}\Bigg) ,&\\
\mathbf C_{10}^* &=  \Bigg(1,1+\frac{19758536784497995373925 \pi ^2}{8997284527737444224}-\frac{5450845004071081252936834855024125 \pi ^4}{1307924593773094121223940118864}&\\
&\phantom{=\Bigg(}\quad\quad +\frac{237746319601404434678217736438860000 \pi ^6}{134952926502386519274273464063983}-\frac{15125685654401280000000 \pi ^8}{109701233401363445369},&\\
&\phantom{=\Bigg(}\frac{19758536784497995373925 \pi ^2}{4498642263868722112},&\\
&\phantom{=\Bigg(}
\frac{19758536784497995373925 \pi ^2}{8997284527737444224}
+\frac{27254225020355406264684174275120625 \pi ^4}{1307924593773094121223940118864}&\\
&\phantom{=\Bigg(}\quad\quad
-\frac{832112118604915521373762077536010000 \pi ^6}{134952926502386519274273464063983}
+\frac{50418952181337600000000 \pi ^8}{109701233401363445369},&\\
&\phantom{=\Bigg(}
\frac{16352535012213243758810504565072375 \pi ^4}{653962296886547060611970059432},&\\
&\phantom{=\Bigg(}
\frac{5450845004071081252936834855024125 \pi ^4}{653962296886547060611970059432}
+\frac{1664224237209831042747524155072020000 \pi ^6}{134952926502386519274273464063983}
-\frac{70586533053872640000000 \pi ^8}{109701233401363445369},&\\
&\phantom{=\Bigg(}
\frac{1426477917608426608069306418633160000 \pi ^6}{134952926502386519274273464063983},&\\
&\phantom{=\Bigg(}
\frac{356619479402106652017326604658290000 \pi ^6}{134952926502386519274273464063983}
+\frac{75628428272006400000000 \pi ^8}{109701233401363445369},&\\
&\phantom{=\Bigg(}\frac{50418952181337600000000 \pi ^8}{109701233401363445369},\frac{10083790436267520000000 \pi ^8}{109701233401363445369}\Bigg).&
\end{flalign*}
}
\end{theorem}
Observe that the first entry of each vector is $C_{d,0}^* = 1$ for all $d\in\N$.  This is trivial because all points $X_1,\ldots,X_n$ are vertices of $P_{n,d}^{-1}$. Yet, in the above table, the constant $1$ appeared as a result of a non-trivial computation of $\bJ_{d,1}(-1/2)$. On the one hand-side, this gives evidence for the correctness of the algorithm. On the other hand, it can be used to give an explicit formula for $\bJ_{d,1}(-1/2)$, as we shall show in the next section.

\subsection{Special cases: \texorpdfstring{$\bJ_{n,1}(1/2)$}{J\_\{n,k\}(1/2)} and \texorpdfstring{$\bJ_{n,1}(-1/2)$}{J\_\{n,k\}(-1/2)}}
There are only few special cases in which we are able to obtain a ``nice'' formula for $\bJ_{n,k}(\beta)$ or $\tilde \bJ_{n,k}(\beta)$. Most notably, in~\cite{kabluchko_poisson_zero} we obtained an explicit formula for $\tilde \bJ_{n,k}(\frac n2)$ which has applications to the expected $f$-vector of the zero cell of the Poisson hyperplane tessellation. By a similar method, it is also possible to derive a combinatorial formula for $\tilde \bJ_{n,k}(\frac{n+1}{2})$, which will be treated elsewhere. In this section, we shall  prove simple formulae for $\bJ_{n,1}(1/2)$ and $\bJ_{n,1}(-1/2)$. Note that the beta distributions with $\beta=1/2$ and $\beta=-1/2$ are natural multidimensional generalizations of the Wigner semicircle and the arcsine distributions, respectively.

\begin{theorem}\label{theo:beta_+1/2}
For every $n\in\N$ we have
\begin{align*}
\bJ_{n,1}(1/2)
&=
\frac{n(n^2+1)(n^2+n+2)\pi}{(n+3)2^{n(2n+1)}} \binom{n+1}{\frac12 (n+1)}^{n-1} \binom{n^2}{\frac12 n^2} \\
&=
\frac{n(n^2+1)(n^2+n+2)}{2^{n+1}(n+3)\pi^{\frac {n-2}{2}}}
\left(\frac{\Gamma(\frac{n+2}{2})}{\Gamma(\frac{n+3}{2})}\right)^{n-1}
\frac{\Gamma(\frac{n^2+1}{2})}{\Gamma(\frac{n^2+2}{2})}.
\end{align*}
\end{theorem}
\begin{proof}
The argument follows essentially the approach sketched by Hug~\cite[pp.~209--210]{hug_rev}.
Consider $N$ i.i.d.\ points uniformly distributed in the unit ball $\bB^d$. Denote their convex hull by $P_{N,d}^0$. As $N\to\infty$, the random polytope $P_{N,d}^0$ approaches the unit ball. In particular, $\E (\Vol_d P_{N,d}^0)$ converges to $\kappa_d$, the volume of $\bB^d$.  The speed of convergence has been identified by Wieacker~\cite{wieacker_dipl}; see also~\cite{affentranger} for similar results on general beta polytopes and~\cite{affentranger_exact,beta_polytopes_temesvari} for exact formulae for the expected volume.  In particular, it is known that
\begin{equation}\label{eq:J_1/2_vol_diff}
\kappa_d - \E \Vol_d (P_{N,d}^0) \sim  \frac{d\kappa_d}{2d!} \frac{d+1}{d+3} \Gamma\left(\frac{d^2+1}{d+1}+2\right) \left(\frac{2\sqrt{\pi} \Gamma(\frac{d+3}{2})}{\Gamma(\frac{d+2}{2})}\right)^{\frac{2}{d+1}} \cdot N^{-\frac{2}{d+1}},
\end{equation}
as $N\to\infty$; see, for example Corollary~1 on page~366 of~\cite{affentranger} and the formula for $c_5$ on page 378, with $q=0$.
The left-hand side is closely related to the expected number of vertices of $P_{N,d}^0$ via the Efron identity which states that
\begin{equation}\label{eq:J_1/2_efron}
\E f_0 (P_{N,d}^0) = N\cdot (\kappa_d-\E \Vol_d (P_{N-1,d}^0))/\kappa_d.
\end{equation}
Indeed, the $N$-th point is a vertex of $P_{N,d}^0$ if and only if it is outside the convex hull of the remaining $N-1$ points. If we condition on the first $N-1$ points, then the probability that the last point is a vertex is $(\kappa_d- \Vol_d (P_{N-1,d}^0))/\kappa_d$. Taking expectations proves Efron's identity.  From~\eqref{eq:J_1/2_vol_diff} and~\eqref{eq:J_1/2_efron} we deduce that
\begin{equation} \label{eq:J_1/2_E_f0_1}
\E f_0 (P_{N,d}^0) \sim  \frac{d}{2d!} \frac{d+1}{d+3} \Gamma\left(\frac{d^2+1}{d+1}+2\right) \left(\frac{2\sqrt{\pi} \Gamma(\frac{d+3}{2})}{\Gamma(\frac{d+2}{2})}\right)^{\frac{2}{d+1}} \cdot N^{\frac{d-1}{d+1}},
\end{equation}
as $N\to\infty$.
On the other hand, we know from~\eqref{eq:def_C_d_k} and~\eqref{eq:Limitf_kBall} (where we take $k=0$) that
\begin{align}
\E f_0(P_{N,d}^0)
&\sim
C_{d,0} N^{{d-1\over d+1}}\notag\\
&=
{2\pi^{d(d-1)\over 2(d+1)}\over (d+1)!}
{\Gamma(1+{d^2\over 2})\Gamma({d^2+1\over d+1})\over\Gamma({d^2+1\over 2})}\left({(d+1)\Gamma({d+1\over 2})\over \Gamma(1+{d\over 2})}\right)^{d^2+1\over d+1} \bJ_{d,1}(1/2)\cdot   N^{{d-1\over d+1}}, \label{eq:J_1/2_E_f0_2}
\end{align}
as $N\to\infty$.
Equating the constants on the right-hand sides of~\eqref{eq:J_1/2_E_f0_1} and~\eqref{eq:J_1/2_E_f0_2}, resolving w.r.t.\ $\bJ_{d,1}(1/2)$ and simplifying, we arrive at the second formula stated in Theorem~\ref{theo:beta_+1/2}. The equivalence of both formulae is easily shown using the identity
\begin{equation}\label{eq:legendre}
\binom{z}{\frac 12 z} = \frac{2^z\Gamma(\frac{z+1}{2})}{\sqrt \pi \Gamma(\frac{z+2}{2})},
\end{equation}
which is equivalent to the Legendre duplication formula for the Gamma function.
\end{proof}

\begin{theorem}\label{theo:beta_-1/2}
For every $n\in\{2,3,\ldots\}$ we have
\begin{align*}
\bJ_{n,1}(-1/2)
&=
2^{1-n} n\binom {n-1}{\frac 12 (n-1)}^{n-1} \binom{(n-1)^2}{\frac12 (n-1)^2}^{-1}\\
&=
\frac{n(n-1)^2}{2^{n}\pi^{(n-2)/2}}
\left(\frac{\Gamma(\frac{n}{2})}{\Gamma(\frac{n+1}{2})}\right)^{n-1}
\frac{\Gamma(\frac{(n-1)^2}{2})}{\Gamma(\frac{(n-1)^2+1}{2})}.
\end{align*}
\end{theorem}
We shall give two independent proofs. The first one is based on~\eqref{eq:reitzner_const_sphere} (which, as was explained above, generalizes~\eqref{eq:C_star_d_d-1} obtained independently in~\cite{affentranger} and~\cite{buchta_mueller_tichy}). The second proof relies, among other ingredients, on a formula due to Kingman~\cite{kingman_secants}. The fact that all these formulae lead to the same result can be viewed as an additional evidence for their correctness.

\begin{proof}[First proof of Theorem~\ref{theo:beta_-1/2}]
Recall that $P_{N,d}^{-1}$ is the convex hull of $N$ i.i.d.\ points having the uniform distribution on $\bS^{d-1}$. By a formula derived in~\cite{beta_polytopes}, we have
$$
\lim_{N\to\infty} \frac 1N  \E f_k(P_{N,d}^{-1})
= {2^d\pi^{{d\over 2}-1}\over d(d-1)^2} \bJ_{d,k+1}\left(-{1\over 2}\right){\Gamma(1+{d(d-2)\over 2})\over\Gamma({(d-1)^2\over 2})}\left({\Gamma({d+1\over 2})\over\Gamma({d\over 2})}\right)^{d-1}.
$$
On the other hand, in the special case when $k=0$ we trivially have $f_0(P_{N,d}^{-1})=N$ a.s.\ since every point is a vertex. Hence, the right-hand side equals $1$ if $k=0$, which yields
$$
\bJ_{d,1}(-1/2) =  {d(d-1)^2 \over 2^d\pi^{{d\over 2}-1}} {\Gamma({(d-1)^2\over 2})\over\Gamma(1+{d(d-2)\over 2})} \left({\Gamma({d\over 2})\over\Gamma({d+1\over 2})}\right)^{d-1}.
$$
Replacing $d$ by $n$ completes the proof of the second formula stated in Theorem~\ref{theo:beta_-1/2}. The equivalence to the first formula follows from Legendre's duplication formula~\eqref{eq:legendre}.
\end{proof}

The second proof of Theorem~\ref{theo:beta_-1/2} uses the following observation of Feldman and Klain~\cite{feldman_klain}. It can be viewed as a special case of a more general result that has been obtained earlier by Affentranger and Schneider~\cite{AS92}.
\begin{theorem}[Feldman and Klain]\label{theo:feldman_klain}
Let $S=[x_0,\ldots,x_d]\subset \R^d$ be a $d$-dimensional simplex. Let $U$ be a random vector uniformly distributed on the unit sphere $\bS^{d-1}$ and denote by $\Pi=\Pi_{U^\bot}$ the orthogonal projection onto the orthogonal complement of $U$.
Then, the sum of solid angles at all vertices of $S$ is given by
$$
s_0(S) = \frac 12 \P[\, \Pi S \text{ is a $(d-1)$-dimensional simplex} \,].
$$
\end{theorem}

\begin{proof}[Second proof of Theorem~\ref{theo:beta_-1/2}]
Let $X_0,\ldots,X_d$, where $d = n-1$, be i.i.d.\ random points in $\R^d$ with probability density $f_{d,-1/2}$. Independently of these points, let $U$ be a uniform random point on the sphere $\bS^{d-1}$.
Consider an orthogonal projection $\Pi$ of the simplex $[X_0,\ldots,X_d]$ onto a random, uniformly distributed, hyperplane $L:= U^\bot$.  Then, it follows from Theorem~\ref{theo:feldman_klain} and Fubini's formula that
$$
\bJ_{n,1}(-1/2) = \frac {d+1} 2 \P[\Pi X_0 \text{ is a not vertex of } [\Pi X_0,\ldots,\Pi X_d]].
%=  \frac {1-p} 2,
$$
%where
%$$
%p:= \P[\Pi X_0 \text{ is a vertex of } [\Pi X_0,\ldots,\Pi X_d]].
%$$
Let us compute the probability on the right-hand side. Let $I_L:L\to \R^{d-1}$ be an isometry with $I_L(0)=0$. By the projection property of the beta densities (see~\cite[Lemma~4.4]{beta_polytopes_temesvari}) the points
$$
Y_0:=I_L(\Pi X_0),\ldots, Y_d := I_L(\Pi X_d),
$$
have the density $f_{d-1,0}$. That is, these points are uniformly distributed in the unit ball $\bB^{d-1}$. Clearly, these points are i.i.d.  We have
\begin{align*}
\P[\Pi X_0 \text{ is a not vertex of } [\Pi X_0,\ldots,\Pi X_d]]
&=
\P[Y_0 \text{ is a not vertex of } [Y_0,\ldots,Y_d]]\\
&=
\P[Y_0 \in [Y_1,\ldots,Y_d]]\\
&=
\frac 1{\kappa_{d-1}} \E \Vol_{d-1} [Y_1,\ldots,Y_d],
\end{align*}
where the last equality is the Efron identity obtained by conditioning on $Y_1,\ldots,Y_d$ and recalling that $Y_0$ is uniformly distributed in $\bB^{d-1}$.
A formula for the expected volume on the right-hand side is well known from the work of Kingman~\cite[Theorem~7]{kingman_secants}:
$$
\E \Vol_{d-1} [Y_1,\ldots,Y_d]
=
\kappa_{d-1} \binom d{d/2}^d \binom {d^2}{d^2/2}^{-1} 2^{1-d}
=
\kappa_{d-1} \frac{(d+1) d^2}{2^{d}\pi^{(d-1)/2}}
\left(\frac{\Gamma(\frac{d+1}{2})}{\Gamma(\frac{d+2}{2})}\right)^{d}
\frac{\Gamma(\frac{d^2}{2})}{\Gamma(\frac{d^2+1}{2})},
$$
where the second equality can be verified using the duplication formula for the Gamma function.
Taking everything together and recalling that $d=n-1$ completes the proof.
\end{proof}

%%%%%%
%%Here I removed the section "Speculations on reciprocity", see after \end{document}

\section{Proofs: Formulae for internal angles}

\subsection{Notation and facts from stochastic geometry}\label{sec:notation}
Let us first introduce the necessary notation, referring to the book by Schneider and Weil~\cite{SW08} for an extensive account of stochastic geometry. A \textit{polyhedral cone} (or just a \textit{cone}) $C\subset \R^d$ is an intersection of finitely many closed halfspaces whose boundaries pass through the origin. The \textit{solid angle} of $C$ is defined as
$$
\alpha (C) = \P [U \in C],
$$
where $U$ is a random vector having the uniform distribution on the unit sphere of the smallest linear subspace containing $C$. For example, the angle of $\R^d$ is $1$, whereas the angle of any half-space is $1/2$.
%Let $X_0,\ldots,X_d$ be random points in the $d$-dimensional Euclidean space $\R^d$ sampled independently according to the uniform distribution on the unit sphere  or the unit ball .
Let $P\subset \R^d$  be a $d$-dimensional convex polytope.  Denote by $\cF_k(P)$ the set of its $k$-dimensional faces, where $k\in \{0,1,\ldots, d\}$. The set of all faces of $P$ is denoted by $\cF_{\bullet}(P) = \cup_{k=0}^d \cF_k(P)$.  The \textit{tangent cone} of $P$ at its face $F\in \cF_k(P)$ is defined as
$$
T(F,P) := \{y\in \R^d\colon  \exists \eps>0 \text{ such that } f_0 + \eps y \in P\}
$$
where $f_0$ is any point in the relative interior of $F$, defined as the interior of $F$ taken with respect to its affine hull.
The \textit{internal angle} of $P$ at its face $F\in \cF_k(P)$ is defined by
$$
\beta(F,P) := \alpha(T(F,P)).
%\P [U \in T(F,P)],
$$
%where $U$ is a random vector having the uniform distribution on the unit sphere $\bS^{d-1}$.
The \textit{normal} or \textit{external cone} of $F$ is defined as the polar cone of $T(F,P)$, that is
$$
N(F,P) = \{z\in \R^d\colon  \langle z,y \rangle\leq 0 \text{ for all } y\in T(F,P)\}.
$$
The \textit{normal} of \textit{external angle} of $P$ at its face $F\in \cF_k(P)$ is defined by
$$
\gamma(F,P) := \alpha(N(F,P)).
$$
By convention, $\beta(P,P) = \gamma(P,P) = 1$.

For a polyhedral cone $C\subset \R^d$ we denote by  $\upsilon_{0}(C),\ldots,\upsilon_d(C)$  its \textit{conic intrinsic volumes}. There are various equivalent definitions of these quantities, see~\cite{ALMT14,AmelunxenLotzDCG17} and~\cite[Section~6.5]{SW08}. For example, we have
$$
\upsilon_j(C) = \sum_{F\in \cF_j(C)}  \alpha(F) \gamma(F, C),
\qquad
j\in \{0,\ldots,d\}.
$$
It is known, see~\cite[Theorem~6.5.5]{SW08} or~\cite[Equation~(5.1)]{ALMT14}, that for every cone $C\subset \R^d$,
\begin{equation}\label{eq:upsilon_j_sum}
\sum_{j=0}^d \upsilon_j(C) =1.
\end{equation}
Also, the \textit{Gauss-Bonnet relation}, see~\cite[Theorem~6.5.5]{SW08} or~\cite[Equation~(5.3)]{ALMT14}, states that
\begin{equation}\label{eq:upsilon_j_gauss_bonnet}
\sum_{j=0}^d (-1)^j \upsilon_j(C) = 0
\end{equation}
for every $d$-dimensional polyhedral cone $C$ that is not a linear subspace.

%Finally, denote the sum of angles of $P$ at its $k$-dimensional faces by
%$$
%s_k(P) = \sum_{F\in \cF_k(P)} \beta(F,P),
%$$
%for $k \in \{0,\ldots,d\}$. Note that we always have $s_d(P)=1$ and that $s_{d-1} (P)$ is $1/2$ times the number of facets of $P$ (since the tangent cones at the facets are half-spaces having the angle $1/2$).

\subsection{Proof of Proposition~\ref{prop:relations}}\label{subsec:proof_relations}
Consider the $(n-1)$-dimensional random  simplices
$$
P_{n,n-1}^\beta:= [X_1,\ldots,X_{n}] \text{ and }
\tilde P_{n,n-1}^\beta:= [\tilde X_1,\ldots,\tilde X_{n}]
$$
where $X_1,\ldots,X_n$ (respectively, $\tilde X_1,\ldots,\tilde X_n$) are independent random points in $\R^{n-1}$ with probability density $f_{n-1,\beta}$ (respectively, $\tilde f_{n-1,\beta}$).  Let $G$ (respectively, $\tilde G$) be a $k$-vertex face of $P_{n,n-1}^\beta$ (respectively, $\tilde P_{n,n-1}^\beta$). Without loss of generality, we can take $G=[X_1,\ldots,X_k]$ and $\tilde G= [\tilde X_1,\ldots,\tilde X_k]$. The tangent cones of these simplices at this face  are defined as
\begin{align*}
T_{n,k}^\beta
&:=
\{v\in\R^{n-1}: \text{ there exists } \eps>0 \text{ such that }  g_0 + \eps v\in  P_{n,n-1}^\beta\},\\
\tilde T_{n,k}^\beta
&:=
\{v\in\R^{n-1}: \text{ there exists } \eps>0 \text{ such that } \tilde g_0 + \eps v\in \tilde P_{n,n-1}^\beta\},
\end{align*}
where $g_0$ (respectively, $\tilde g_0$) is any point in the relative interior of $G$ (respectively $\tilde G$).
The expected conic intrinsic volumes of the tangent cones $T_{n,k}^\beta$ and $\tilde T_{n,k}^\beta$ were computed in~\cite[Theorems~1.12 and~1.18]{beta_polytopes}. Namely, it was shown there that for all $k\in \{1,\ldots,n-1\}$ and $j\in \{k-1,\ldots,n-1\}$ we have
\begin{align}
\E \upsilon_j(T_{n,k}^\beta)
&=
\frac 1 {\binom {n}{k}} \bI_{n,j+1}(2\beta+n-1) \tilde \bJ_{j+1,k}\left(\beta + \frac{n-1-j}{2}\right),\label{eq:E_intrinsic_tangent_cone}\\
\E \upsilon_j(\tilde T_{n,k}^\beta)
&=
\frac 1 {\binom {n}{k}}\tilde \bI_{n,j+1}(2\beta-n+1) \tilde \bJ_{j+1,k}\left(\beta - \frac{n-1-j}{2}\right).\label{eq:E_intrinsic_tangent_cone_tilde}
\end{align}
For $j\notin \{k-1,\ldots,n-1\}$ we have $\upsilon_j(T_{n,k}^\beta) = \upsilon_j(\tilde T_{n,k}^\beta)=0$, which is due to the fact that the tangent cones contain the $(k-1)$-dimensional linear subspace spanned by $X_1-g_0,\ldots,X_k-g_0$ (respectively, $\tilde X_1-\tilde g_0,\ldots,\tilde X_k-\tilde g_0$).
Applied to the tangent cones $T_{n,k}^\beta$ and $\tilde T_{n,k}^\beta$,  Relations~\eqref{eq:upsilon_j_sum} and~\eqref{eq:upsilon_j_gauss_bonnet} read as
$$
\sum_{j=k-1}^{n-1} \upsilon_j(T_{n,k}^\beta)
=
\sum_{j=k-1}^{n-1} \upsilon_j(\tilde T_{n,k}^\beta)
=
1,
\qquad
\sum_{j=k-1}^{n-1} (-1)^j \upsilon_j(\tilde T_{n,k}^\beta)
=
\sum_{j=k-1}^{n-1} (-1)^j \upsilon_j(\tilde T_{n,k}^\beta)
=
0.
$$
Taking the expectation and applying~\eqref{eq:E_intrinsic_tangent_cone} and~\eqref{eq:E_intrinsic_tangent_cone_tilde}, we arrive at the required relations~\eqref{eq:relation_I_J_1}, \eqref{eq:relation_I_J_2}, \eqref{eq:relation_I_J_1_tilde}, \eqref{eq:relation_I_J_2_tilde}.
\hfill $\Box$

\begin{remark}
It is possible to obtain another proof of Proposition~\ref{prop:relations} using McMullen's non-linear angle-sum relations~\cite{mcmullen,mcmullen_polyhedra}. These state that for every face $F\in \cF_\bullet(P)$ of an arbitrary  polytope $P$,
\begin{align*}
&\sum_{H\in \cF_\bullet(P): F\subset H \subset P} \beta(F,H) \gamma(H,P) = 1,\\
&\sum_{H\in \cF_\bullet(P): F\subset H \subset P} (-1)^{\dim H - \dim P} \beta(F,H) \gamma(H,P) = \delta_{F,P},
\end{align*}
where $\delta_{F,P} = 1$ if $F=P$, and $\delta_{F,P}=0$, otherwise. Applied to $P = P_{n,n-1}^\beta = [X_1,\ldots,X_{n}]$ and $F=[X_1,\ldots,X_k]$, the first relation reads
\begin{equation}\label{eq:mcmullen_for_beta_simpl}
\sum_{m=k}^n \binom{n-k}{m-k}\beta([X_1,\ldots,X_k],[X_1,\ldots,X_m]) \gamma([X_1,\ldots,X_m],[X_1,\ldots,X_n]) = 1.
\end{equation}
To prove Equation~\eqref{eq:relation_I_J_1} of Proposition~\ref{prop:relations}, one is tempted to take the expectation of this relation. This has to be done with care because the relation is non-linear.
First of all, by Theorem~\ref{theo:external} we have
$$
\E \gamma([X_1,\ldots,X_m],[X_1,\ldots,X_n]) = I_{n,m} (2\beta + n-1).
$$
The so-called canonical decomposition of beta distributions, see~\cite{ruben_miles} or~\cite[Theorem~3.3]{beta_polytopes}, implies that the random variables $\gamma([X_1,\ldots,X_m],[X_1,\ldots,X_n])$ and  $\beta([X_1,\ldots,X_k],[X_1,\ldots,X_m])$ are stochastically independent; see~\cite[Theorem~1.6]{beta_polytopes} for the statement and~\cite[Section~4.1]{beta_polytopes} for the proof.  Finally, Theorem~4.1 in~\cite{beta_polytopes} with $d=m-1$, $\ell = n-m$ implies that
$$
\E \beta([X_1,\ldots,X_k],[X_1,\ldots,X_m])
=
J_{m,k} \left(\beta + \frac{n-m}{2}\right).
$$
Observe that on the right-hand side we have a quantity different from $J_{m,k}(\beta)$ since the points $X_1,\ldots,X_m$ are in $\R^{n-1}$ and do not form a full-dimensional simplex, so that we cannot directly apply the definition of $J_{m,k}(\beta)$. Taking the expectation of~\eqref{eq:mcmullen_for_beta_simpl} and using the above facts, we obtain
$$
\sum_{m=k}^n \binom{n-k}{m-k} I_{n,m} (2\beta + n-1) J_{m,k} \left(\beta + \frac{n-m}{2}\right) = 1.
$$
Recalling that $\bI_{n,k}(\alpha) = \binom nk I_{n,k}(\alpha)$ and $\bJ_{n,k}(\beta) = \binom nk J_{n,k}(\beta)$, we arrive at~\eqref{eq:relation_I_J_1}.
The proofs of~\eqref{eq:relation_I_J_2}, \eqref{eq:relation_I_J_1_tilde}, \eqref{eq:relation_I_J_2_tilde} are similar.
\hfill $\Box$
\end{remark}

\subsection{Proof of Theorem~\ref{theo:formula_J_1}}
We use induction over $n$. The claim is true for $n=k=1$ since $\bJ_{1,1}(\beta) = 1$ and $\bI_{1,1}(2\beta) = 1$. Assume that, for some $n\geq 2$, the claim is true for all quantities $\bJ_{m,k}(\gamma)$ with $m\in \{1,\ldots,n-1\}$, $k\in \{1,\ldots,m\}$, $\gamma\geq -1$. In particular, we have
$$
\bJ_{m,k}\left(\beta+\frac {n-m}2\right)
=
\sum_{\ell=0}^{m-k} (-1)^\ell \sum_{m=m_0>m_1>\ldots>m_\ell\geq k} \bI_{m, m_1}(2\beta+n-1) \ldots \bI_{m_{\ell-1}, m_\ell}(2\beta + n-1) \binom {m_\ell}{k}.
$$
By~\eqref{eq:alg_J_1}, we have
$$
\bJ_{n,k}(\beta)
=
\binom {n}{k} - \sum_{m=k}^{n-1}  \bI_{n,m}(2\beta+n-1) \bJ_{m,k}\left(\beta+\frac {n-m}2\right).
$$
Using the induction assumption, we obtain
\begin{align*}
\bJ_{n,k}(\beta)
&=
\binom {n}{k} - \sum_{m=k}^{n-1}  \sum_{\ell=0}^{m-k} (-1)^\ell \sum_{m=m_0>m_1>\ldots>m_\ell\geq k}\\
&\qquad  \bI_{n,m}(2\beta+n-1) \bI_{m, m_1}(2\beta+n-1) \ldots \bI_{m_{\ell-1}, m_\ell}(2\beta + n-1) \binom {m_\ell}{k}\\
&=
\binom {n}{k} - \sum_{\ell'=1}^{n-k} (-1)^{\ell'-1} \sum_{n=n_0>n_1>\ldots>n_{\ell'}\geq k}\\ &\qquad
\bI_{n,n_1}(2\beta+n-1) \bI_{n_1, n_2}(2\beta+n-1) \ldots \bI_{n_{\ell'-1}, n_{\ell'}}(2\beta + n-1) \binom {n_{\ell'}}{k},
\end{align*}
where we used the index shift $\ell' = \ell +1$, $(n_1,\ldots,n_{\ell'}) = (m_0,\ldots, m_\ell)$. Note that $\binom nk$ can be interpreted as the term corresponding to $\ell'=0$. This completes the induction.
\hfill $\Box$

\vspace*{1mm}
Theorem~\ref{theo:formula_J_2} can be established analogously by using Relation~\eqref{eq:alg_J_2} instead of~\eqref{eq:alg_J_1}.
%%%%%%
%%%Here a subsection with the proof of the invariance of expected internal angles

\section{Proofs: Arithmetic properties}
In this section we prove Theorems~\ref{theo:I_arithm} and~\ref{theo:arithm_J}.  The proofs of Theorems~\ref{theo:I_arithm_tilde} and~\ref{theo:arithm_J_tilde}, being analogous to the proofs of Theorems~\ref{theo:I_arithm} and~\ref{theo:arithm_J}, are omitted.
%We prove only Theorem~\ref{theo:arithm_J} since the proof of Theorem~\ref{theo:arithm_J_tilde} is analogous.

\subsection{Proof of Theorems~\ref{theo:I_arithm} and~\ref{theo:arithm_J}}
Recall from Section~\ref{sec:relations} that we can express $\bJ_{n,k}(\beta)$ through the quantities of the form
%Recall from Section~\ref{subsec:external} that
$$
\bI_{n,k}(\alpha)
= \binom nk \int_{-\pi/2}^{+\pi/2} c_{1,\frac{\alpha k - 1}{2}} (\cos \varphi)^{\alpha k} \left(\int_{-\pi/2}^\varphi c_{1,\frac{\alpha-1}{2}}(\cos \theta)^{\alpha} \,\dd \theta \right)^{n-k} \, \dd \varphi,
\qquad \alpha\geq 0,
$$
where
$$
c_{1,\beta}= \frac{ \Gamma\left( \frac{3}{2} + \beta\right) }{\sqrt \pi \Gamma(\beta+1)},
\qquad \beta>-1.
$$
In Propositions~\ref{prop:I_rational} and~\ref{prop:I_polynomial_in_pi} we shall establish the arithmetic properties of $\bI_{n,k}(\alpha)$ for integer $\alpha\geq 0$. Taken together, these propositions yield Theorem~\ref{theo:I_arithm}.
\begin{lemma}\label{lem:c_arithm}
Let $\beta > -1$.
\begin{itemize}
\item[(a)] If $\beta$ is integer, then $c_{1,\beta}$ is rational.
\item[(b)] If $\beta$ is half-integer, then $c_{1,\beta}$ is a rational multiple of $\pi^{-1}$.
\end{itemize}
\end{lemma}
\begin{proof}
Just recall the following two facts: (i) $\Gamma(x)$ is integer if $x>0$ is integer.  (ii) $\Gamma(x)$ is a rational multiple of $\Gamma(\frac 12) = \sqrt \pi$ if $x>0$ is half-integer.
\end{proof}

\begin{lemma}\label{lem:integral_cos_power}
If $k\geq 1$ is an odd integer, then $\int_{-\pi/2}^\varphi (\cos \theta)^k \dint \theta$ can be represented as a linear combination of the functions $1, \sin \varphi, \sin (3\varphi), \ldots, \sin (k\varphi)$ with rational coefficients.
\end{lemma}
\begin{proof}
We have
$$
(\cos \theta)^k = \left(\frac {\eee^{\ii \theta} + \eee^{-\ii \theta}}{2}\right)^k
=
\sum_{m = \pm 1, \pm 3,\ldots} q_m \eee^{\ii m \theta}
=
\sum_{m = 1,3,\ldots} 2q_m \cos (m\theta)
$$
for some rational numbers $q_m$ satisfying $q_m=q_{-m}$ and vanishing for $m>k$. By integration it follows that
$$
\int_{-\pi/2}^\varphi (\cos \theta)^k \dint \theta
=
\sum_{m = 1,3,\ldots} 2q_m \int_{-\pi/2}^\varphi  \cos (m\theta) \dint \theta
=
\sum_{m = 1,3,\ldots} \frac{2q_m}m (\sin (m\varphi)-\sin (-m\pi/2)),
$$
which proves the claim.
\end{proof}

\begin{lemma}\label{lem:integral_cos_power_even}
If $k\geq 0$ is an even integer, then $\int_{-\pi/2}^\varphi (\cos \theta)^k \dint \theta$ can be represented as a linear combination of the functions $\pi, \varphi, \sin (2\varphi), \sin (4\varphi), \ldots, \sin (k\varphi)$ with rational coefficients.
\end{lemma}
\begin{proof}
We have
$$
(\cos \theta)^k = \left(\frac {\eee^{\ii \theta} + \eee^{-\ii \theta}}{2}\right)^k
=
\sum_{m = 0, \pm 2, \pm 4,\ldots} q_m \eee^{\ii m \theta}
=
q_0 + \sum_{m = 2,4,\ldots} 2q_m \cos (m\theta)
$$
for some rational numbers $q_m$ satisfying $q_m=q_{-m}$ and vanishing for $m>k$. By integration it follows that
\begin{align*}
\int_{-\pi/2}^\varphi (\cos \theta)^k \dint \theta
&=
q_0 \cdot \left(\varphi + \frac \pi 2\right) + \sum_{m = 2,4,\ldots} 2q_m \int_{-\pi/2}^\varphi  \cos (m\theta) \dint \theta\\
&=
q_0 \cdot \left(\varphi + \frac \pi 2\right) + \sum_{m = 2,4,\ldots} \frac{2q_m}m (\sin (m\varphi)-\sin (-m\pi/2)),
\end{align*}
which proves the claim since $\sin (-m\pi/2) = 0$ for even $m$.
\end{proof}

\begin{proposition}\label{prop:I_rational}
If $\alpha\geq 1$ is an odd integer, then $\bI_{n,k}(\alpha)$ is rational for all $n\in\N$, $k\in \{1,\ldots,n\}$.
\end{proposition}
\begin{proof}
Note that $c_{1,\frac{\alpha-1}{2}}$ is rational by Lemma~\ref{lem:c_arithm}.  Using Lemma~\ref{lem:integral_cos_power} and the formula $\sin t = (\eee^{\ii t} - \eee^{-\ii t})/(2\ii)$ we can write
$$
\int_{-\pi/2}^\varphi c_{1,\frac{\alpha-1}{2}}(\cos \theta)^{\alpha} \,\dd \theta
= a + \sum_{m=1,3,\ldots} a_m \sin (m\varphi)
= a + \sum_{m=\pm 1, \pm 3,\ldots} a_m' \ii  \eee^{\ii m \varphi},
$$
for some $a,a_m,a_m'\in\bQ$. The sums in the above equality, as well as all sums in this proof, have only finitely many non-zero terms.

\vspace*{2mm}
\noindent
\textsc{Case 1:} Let $k\in \{1,\ldots,n\}$ be odd. Then, $c_{1,\frac{\alpha k - 1}{2}}$ is rational by Lemma~\ref{lem:c_arithm}, and we can write
$$
c_{1,\frac{\alpha k - 1}{2}} (\cos \varphi)^{\alpha k}
=
c_{1,\frac{\alpha k - 1}{2}} \left(\frac{\eee^{\ii \varphi} + \eee^{-\ii \varphi} }{2}\right)^{\alpha k}
=
\sum_{\ell=\pm 1, \pm 3,\ldots} b_\ell \eee^{\ii \ell \varphi}
$$
with some rational numbers $b_\ell$. Taking everything together, we arrive at
$$
\bI_{n,k}(\alpha)
= \binom nk \int_{-\pi/2}^{+\pi/2} \left(\sum_{\ell=\pm 1, \pm 3,\ldots} b_\ell \eee^{\ii \ell \varphi}\right) \left(a + \sum_{m=\pm 1, \pm 3,\ldots} a_m' \ii \eee^{\ii m \varphi}\right)^{n-k} \, \dd \varphi.
$$
When multiplying out the terms under the integral sign, we obtain a finite $\bQ$-linear combination of the terms of the form $\eee^{i s \varphi}$ (with odd $s$) and  $\ii \eee^{is \varphi}$ (with even $s$).   The integral of a term of the former type is a rational number since
$$
\int_{-\pi/2}^{+\pi/2} \eee^{i s \varphi}\dd \varphi = \frac{1}{\ii s} (\eee^{i s \pi/2} - \eee^{-i s \pi/2})\in \bQ,
\qquad s\in \{\pm 1, \pm 3, \ldots\}.
$$
The integrals of the terms of the latter type, with $s\neq 0$,  are also rational since
$$
\int_{-\pi/2}^{+\pi/2} \ii \eee^{i s \varphi}\dd \varphi = \frac{1}{s} (\eee^{i s \pi/2} - \eee^{-i s \pi/2})\in \bQ,
\qquad s\in \{\pm 2, \pm 4, \ldots\}.
$$
Finally, the term $\ii \eee^{\ii 0 \varphi}$ must have coefficient $0$ since its integral is purely imaginary and we know a priori that $\bI_{n,k}(\alpha)$ is real.
Hence, $\bI_{n,k}(\alpha)$ is rational.

\vspace*{2mm}
\noindent
\textsc{Case 2:} Let $k\in \{1,\ldots,n\}$ be even. Then, $\alpha k$ is also even and $c_{1,\frac{\alpha k - 1}{2}}$ is a rational multiple of $1/\pi$ by Lemma~\ref{lem:c_arithm}. We can write
$$
c_{1,\frac{\alpha k - 1}{2}} (\cos \varphi)^{\alpha k}
=
c_{1,\frac{\alpha k - 1}{2}} \left(\frac{\eee^{\ii \varphi} + \eee^{-\ii \varphi} }{2}\right)^{\alpha k}
=
\frac 1 \pi \sum_{\ell= 0, \pm 2, \pm 4,\ldots} b_\ell \eee^{\ii \ell \varphi}
$$
with some rational numbers $b_\ell$, where the sum contains only finitely many non-zero terms.  Taking everything together, we arrive at
$$
\bI_{n,k}(\alpha)
= \binom nk \int_{-\pi/2}^{+\pi/2} \left(\frac 1 \pi \sum_{\ell= 0, \pm 2, \pm 4,\ldots} b_\ell \eee^{\ii \ell \varphi} \right) \left(a + \sum_{m=\pm 1, \pm 3,\ldots} a_m' \ii \eee^{\ii m \varphi}\right)^{n-k} \, \dd \varphi.
$$
When multiplying out the terms under the sign of the integral, we obtain a finite $\bQ$-linear combination of the terms of the form $\pi^{-1} \eee^{i s \varphi}$ (with even $s$) and  $\ii \pi^{-1} \eee^{i s \varphi}$ (with odd $s$). The integral of the term $\pi^{-1}\eee^{i 0 \varphi}$ is $1$.  By the same analysis as in Case~1, the integrals of all terms with $s\neq 0$ are purely imaginary and hence must cancel since we know a priori that $\bI_{n,k}(\alpha)$ is real.
Hence, $\bI_{n,k}(\alpha)$ is rational.
\end{proof}

\begin{proof}[Proof of Theorem~\ref{theo:arithm_J}, Part~(a)]
Let $n\in \N$, $k\in\{1,\ldots,n\}$, and let $\beta\geq -1$ be such that $2\beta + n$ is even. Our aim is to prove that $\bJ_{n,k}(\beta)$ is rational. This is done by induction. The claim is trivial for  $n=1,2,3$.  Assuming that, for some $n\geq 4$, the statement has been established for all $\bJ_{m,k}(\gamma)$ with $m\in \{1,\ldots,n-1\}$, we recall that by~\eqref{eq:alg_J_1},
$$
\bJ_{n,k}(\beta)
=
\binom {n}{k} - \sum_{s=1}^{n-k}  \bI_{n,n-s}(2\beta+n-1) \bJ_{n-s,k}\left(\beta+\frac s2\right).
$$
The numbers $\bI_{n,n-s}(2\beta+n-1)$ are rational by Proposition~\ref{prop:I_rational}, whereas  the terms $\bJ_{n-s,k}(\beta+\frac s2)$ are rational by induction assumption, for all $s\in \{1,\ldots,n-k\}$.
\end{proof}

Next we are going to analyze $\bI_{n,k}(\alpha)$ for even $\alpha\geq 0$. To this end, we need the following
\begin{lemma}\label{lem:term}
Consider the integral $T(s,p) = \frac 1\pi \int_{-\pi/2}^{+\pi/2} \eee^{\ii s \varphi} (\varphi/\pi)^p \dint \varphi$, where $s$ is an even integer, and $p\geq 0$ is integer.
\begin{itemize}
\item[(a)] If $p$ is even, then $T(s,p)$ can be represented as $q_0+ q_2\pi^{-2}+q_4 \pi^{-4}+\ldots+q_{p}\pi^{-p}$ with rational $q_i$'s.
\item[(b)] If $p$ is odd, then $T(s,p)$ can be represented as $\frac \ii \pi (q_0+ q_2\pi^{-2}+q_4 \pi^{-4}+\ldots+q_{p-1}\pi^{-(p-1)})$ with rational $q_i$'s.
    \end{itemize}
\end{lemma}
\begin{proof}
For $s=0$ the statement is trivial since $T(0,p) = 0$ for odd $p$ and $T(0,p) = 2^{-p}/(p+1)$ for even $p$. Let $s\neq 0$ be even.  For $p=0$ we have $T(s,p)=0$. For integer $p\geq 1$ the statement follows by induction using the formula
$$
T(s,p) = \frac 1{\pi  \ii s} \int_{-\pi/2}^{+\pi/2}(\varphi/\pi)^p \dint  \eee^{\ii s \varphi}
=
\left.\left(\frac{(\varphi/\pi)^p}{\pi \ii s}p\eee^{\ii s \varphi}\right) \right|_{\varphi=-\pi/2}^{\varphi=+\pi/2}  +\frac{\ii p}{\pi s} T(s,p-1),
$$
which is obtained by partial integration.
\end{proof}

\begin{proposition}\label{prop:I_polynomial_in_pi}
If  $\alpha\geq 0$ is even, $n\in\N$ and $k\in \{1,\ldots,n\}$, then $\bI_{n,k}(\alpha)$ can be expressed in the form $r_0+r_2\pi^{-2} + r_4\pi^{-4} +\ldots + r_{n-k} \pi^{-(n-k)}$ (if $n-k$ is even) or $r_0+r_2\pi^{-2} + r_4\pi^{-4} +\ldots + r_{n-k-1} \pi^{-(n-k-1)}$ (if $n-k$ is odd), where the $r_i$'s are rational numbers.
\end{proposition}
\begin{proof}
Note that $c_{1,\frac{\alpha-1}{2}}$ is a rational multiple of $1/\pi$ by Lemma~\ref{lem:c_arithm}.  Using Lemma~\ref{lem:integral_cos_power_even} and the formula $\sin t = (\eee^{\ii t} - \eee^{-\ii t})/(2\ii)$ we can write
$$
\int_{-\pi/2}^\varphi c_{1,\frac{\alpha-1}{2}}(\cos \theta)^{\alpha} \,\dd \theta
= a' + \frac{\varphi}{\pi} a'' +  \sum_{m=2,4,\ldots} \frac{a_m'''}{\pi} \sin (m\varphi)
= a' + \frac{\varphi}{\pi} a'' + \sum_{m=\pm 2, \pm 4,\ldots} \frac{a_m}{\pi} \ii  \eee^{\ii m \varphi},
$$
for some $a',a'',a_m''',a_m\in\bQ$.   Recall that $\alpha k$ is even and hence $c_{1,\frac{\alpha k - 1}{2}}$ is a rational multiple of $1/\pi$ by Lemma~\ref{lem:c_arithm}. Thus, we can write
$$
c_{1,\frac{\alpha k - 1}{2}} (\cos \varphi)^{\alpha k}
=
c_{1,\frac{\alpha k - 1}{2}} \left(\frac{\eee^{\ii \varphi} + \eee^{-\ii \varphi} }{2}\right)^{\alpha k}
=
\frac 1 \pi \sum_{\ell= 0, \pm 2, \pm 4,\ldots} b_\ell \eee^{\ii \ell \varphi}
$$
with some rational numbers $b_\ell$, where we recall the convention that the sums contain only finitely many non-zero terms.  Taking everything together, we arrive at
$$
\bI_{n,k}(\alpha)
=
\binom nk \int_{-\pi/2}^{+\pi/2} \left(\frac 1 \pi \sum_{\ell= 0, \pm 2, \pm 4,\ldots} b_\ell \eee^{\ii \ell \varphi} \right) \left( a' + \frac{\varphi}{\pi} a'' + \sum_{m=\pm 2, \pm 4,\ldots} \frac{a_m}{\pi} \ii  \eee^{\ii m \varphi}\right)^{n-k} \, \dd \varphi.
$$
When multiplying everything out, we obtain a representation of $\bI_{n,k}(\alpha)$ as a finite $\bQ$-linear combination of the terms of the form
$$
\frac 1 \pi \left(\frac{\ii}{\pi}\right)^b  \int_{-\pi/2}^{+\pi/2} \eee^{\ii s \varphi} \left(\frac{\varphi}{\pi}\right)^p \dint \varphi = \left(\frac{\ii}{\pi}\right)^b T(s,p),
$$
where $s$ is even, $p\geq 0$ and $b\geq 0$ are integers with $p+b\in \{0,\ldots, n-k\}$.
If both $p$ and $b$ are even, then by Lemma~\ref{lem:term} (a) the term is a $\bQ$-linear combination of $1,\pi^{-2},\pi^{-4}, \ldots, \pi^{-(p+b)}$.
If both $p$ and $b$ are odd, then by Lemma~\ref{lem:term} (b) the term is a $\bQ$-linear combination of $1,\pi^{-2},\pi^{-4}, \ldots, \pi^{-(p+b)}$.
If the parities of $p$ and $b$ differ, then the term is purely imaginary and can be ignored since we a priori know that $\bI_{n,k}(\alpha)$ is real, which implies that all such terms must cancel.
\end{proof}

\begin{proof}[Proof of Theorem~\ref{theo:arithm_J}, Part~(b)]
Let $n\in\N$, $k\in \{1,\ldots,n\}$, and $\beta\geq -1$ be such that $2\beta + n$ is odd. We prove by induction that $\bJ_{n,k}(\beta)$ can be expressed as  $q_0 + q_2 \pi^{-2} + q_4 \pi^{-4} + \ldots + q_{n-k} \pi^{-(n-k)}$ (if $n-k$ is even) or $q_0 + q_2 \pi^{-2} + q_4 \pi^{-4} + \ldots + q_{n-k} \pi^{-(n-k-1)}$ (if $n-k$ is odd), where the $q_{i}$'s are rational. The statement is trivial for $n=1,2,3$. Assume that, for some $n\geq 4$, the statement has been established for $\bJ_{m,k}(\gamma)$ with $m\in \{1,\ldots,n-1\}$.
Recall from~\eqref{eq:alg_J_2} that
$$
\bJ_{n,k}(\beta)
=
\frac 12 \binom {n}{k} - \sum_{s=1}^{\lfloor \frac{n-k}{2}\rfloor}  \bI_{n,n-2s}(2\beta+n-1) \bJ_{n-2s,k}(\beta+s).
$$
By Proposition~\ref{prop:I_polynomial_in_pi}, $\bI_{n,n-2s}(2\beta+n-1)$ can be expressed in the form $r_0+r_2\pi^{-2} + r_4\pi^{-4} +\ldots + r_{2s} \pi^{-2s}$ with rational $r_i$'s. On the other hand, by the induction assumption, we can write $\bJ_{n-2s,k}(\beta+s)$ in the form $q_0' + q_2' \pi^{-2} + q_4' \pi^{-4} + \ldots + q_{n-2s-k}' \pi^{-(n-2s-k)}$ (if $n-k$ is even) or $q_0' + q_2' \pi^{-2} + q_4' \pi^{-4} + \ldots + q_{n-2s-k}' \pi^{-(n-2s-k-1)}$ (if $n-k$ is odd) with rational $q_i'$'s. Multiplying everything out, we obtain the required statement.
\end{proof}

%\end{proof}

%\subsection{Typical cell of the Poisson-Voronoi tessellation}

%\subsection{Reitzner constants}

\section*{Acknowledgement}
Supported by the German Research Foundation under Germany's Excellence Strategy  EXC 2044 -- 390685587, Mathematics M\"unster: Dynamics - Geometry - Structure. The author is grateful to P.~Calka and C.~Th\"ale for pointing out references~\cite{moller,moller_book,hug_affine}, and to D.~Hug for helping to correct an error in Section~\ref{subsec:rand_poly_appr_conv_body}.

%\addcontentsline{toc}{section}{References}
%:Referenzen

\bibliography{angles_bib}

\begin{thebibliography}{36}
\providecommand{\natexlab}[1]{#1}
\providecommand{\url}[1]{\texttt{#1}}
\expandafter\ifx\csname urlstyle\endcsname\relax
  \providecommand{\doi}[1]{doi: #1}\else
  \providecommand{\doi}{doi: \begingroup \urlstyle{rm}\Url}\fi

\bibitem[Affentranger(1988)]{affentranger_exact}
F.~Affentranger.
\newblock The expected volume of a random polytope in a ball.
\newblock \emph{J. of Microscopy}, 151\penalty0 (3):\penalty0 277--287, 1988.

\bibitem[{Affentranger}(1991)]{affentranger}
F.~{Affentranger}.
\newblock {The convex hull of random points with spherically symmetric
  distributions.}
\newblock \emph{{Rend. Semin. Mat., Torino}}, 49\penalty0 (3):\penalty0
  359--383, 1991.

\bibitem[Affentranger and Schneider(1992)]{AS92}
F.~Affentranger and R.~Schneider.
\newblock Random projections of regular simplices.
\newblock \emph{Discrete Comput. Geom.}, 7\penalty0 (1):\penalty0 219--226,
  1992.
\newblock \doi{10.1007/BF02187839}.

\bibitem[Amelunxen and Lotz(2017)]{AmelunxenLotzDCG17}
D.~Amelunxen and M.~Lotz.
\newblock Intrinsic volumes of polyhedral cones: a combinatorial perspective.
\newblock \emph{Discrete Comput. Geom.}, 58\penalty0 (2):\penalty0 371--409,
  2017.

\bibitem[Amelunxen et~al.(2014)Amelunxen, Lotz, McCoy, and Tropp]{ALMT14}
D.~Amelunxen, M.~Lotz, M.~McCoy, and J.~Tropp.
\newblock Living on the edge: Phase transitions in convex programs with random
  data.
\newblock \emph{Inform. Inference}, 3:\penalty0 224--294, 2014.

\bibitem[Buchta and M\"uller(1984)]{buchta_mueller}
C.~Buchta and J.~M\"uller.
\newblock Random polytopes in a ball.
\newblock \emph{J. Appl. Probab.}, 21\penalty0 (4):\penalty0 753--762, 1984.

\bibitem[Buchta et~al.(1985)Buchta, M\"uller, and Tichy]{buchta_mueller_tichy}
C.~Buchta, J.~M\"uller, and R.~F. Tichy.
\newblock Stochastical approximation of convex bodies.
\newblock \emph{Math. Ann.}, 271\penalty0 (2):\penalty0 225--235, 1985.

\bibitem[Calka(2010)]{calka_tess_rev}
P.~Calka.
\newblock Tessellations.
\newblock In \emph{New perspectives in stochastic geometry}, pages 145--169.
  Oxford Univ. Press, Oxford, 2010.

\bibitem[Calka(2013)]{calka_tess_asympt_rev}
P.~Calka.
\newblock Asymptotic methods for random tessellations.
\newblock In \emph{Stochastic geometry, spatial statistics and random fields},
  volume 2068 of \emph{Lecture Notes in Math.}, pages 183--204. Springer,
  Heidelberg, 2013.
\newblock \doi{10.1007/978-3-642-33305-7_6}.
\newblock URL \url{https://doi.org/10.1007/978-3-642-33305-7_6}.

\bibitem[Feldman and Klain(2009)]{feldman_klain}
D.~V. Feldman and D.~A. Klain.
\newblock Angles as probabilities.
\newblock \emph{Amer. Math. Monthly}, 116\penalty0 (8):\penalty0 732--735,
  2009.
\newblock \doi{10.4169/193009709X460868}.
\newblock URL \url{https://doi.org/10.4169/193009709X460868}.

\bibitem[Gilbert(1962)]{gilbert}
E.~N. Gilbert.
\newblock Random subdivisions of space into crystals.
\newblock \emph{Ann. Math. Statist.}, 33:\penalty0 958--972, 1962.
\newblock \doi{10.1214/aoms/1177704464}.
\newblock URL \url{https://doi.org/10.1214/aoms/1177704464}.

\bibitem[Godland et~al.(2020)Godland, Kabluchko, and
  Zaporozhets]{godland_kabluchko_zaporozhets}
T.~Godland, Z.~Kabluchko, and D.~Zaporozhets.
\newblock Angle sums of random polytopes.
\newblock Preprint at http://arxiv.org/abs/2007.02590, 2020.

\bibitem[G{\"o}tze et~al.(2019)G{\"o}tze, Kabluchko, and
  Zaporozhets]{goetze_kabluchko_zaporozhets}
F.~G{\"o}tze, Z.~Kabluchko, and D.~Zaporozhets.
\newblock Grassmann angles and absorption probabilities of {G}aussian convex
  hulls.
\newblock Preprint at http://arxiv.org/abs/1911.04184, 2019.

\bibitem[Hug(1996)]{hug_affine}
D.~Hug.
\newblock Contributions to affine surface area.
\newblock \emph{Manuscripta Math.}, 91\penalty0 (3):\penalty0 283--301, 1996.
\newblock \doi{10.1007/BF02567955}.
\newblock URL \url{https://doi.org/10.1007/BF02567955}.

\bibitem[Hug(2013)]{hug_rev}
D.~Hug.
\newblock Random polytopes.
\newblock In \emph{Stochastic geometry, spatial statistics and random fields},
  volume 2068 of \emph{Lecture Notes in Math.}, pages 205--238. Springer,
  Heidelberg, 2013.
\newblock \doi{10.1007/978-3-642-33305-7_7}.
\newblock URL \url{https://doi.org/10.1007/978-3-642-33305-7_7}.

\bibitem[Kabluchko(2019)]{kabluchko_poisson_zero}
Z.~Kabluchko.
\newblock Expected $f$-vector of the {P}oisson zero polytope and random convex
  hulls in the half-sphere, 2019.
\newblock Preprint at http://arxiv.org/abs/1901.10528.

\bibitem[Kabluchko(2020)]{kabluchko_angles}
Z.~Kabluchko.
\newblock Angle sums of random simplices in dimensions $3$ and $4$.
\newblock \emph{Proc. AMS}, 148\penalty0 (7):\penalty0 3079--3086, 2020.

\bibitem[Kabluchko and Zaporozhets(2018)]{kabluchko_zaporozhets_gauss_simplex}
Z.~Kabluchko and D.~Zaporozhets.
\newblock Angles of the {G}aussian simplex.
\newblock \emph{Zap. Nauchn. Sem. POMI}, 476:\penalty0 79--91, 2018.
\newblock Preprint at http://arxiv.org/abs/1801.08008.

\bibitem[Kabluchko et~al.(2018)Kabluchko, Th{\"a}le, and
  Zaporozhets]{beta_polytopes}
Z.~Kabluchko, C.~Th{\"a}le, and D.~Zaporozhets.
\newblock Beta polytopes and {P}oisson polyhedra: {$f$}-vectors and angles.
\newblock Preprint at http://arxiv.org/abs/1805.01338, 2018.

\bibitem[Kabluchko et~al.(2019{\natexlab{a}})Kabluchko, Marynych, Temesvari,
  and Th{\"a}le]{convex_hull_sphere}
Z.~Kabluchko, A.~Marynych, D.~Temesvari, and C.~Th{\"a}le.
\newblock Cones generated by random points on half-spheres and convex hulls of
  {P}oisson point processes.
\newblock \emph{Probab. Theory Relat. Fields}, 175:\penalty0 1021--1061,
  2019{\natexlab{a}}.
\newblock URL \url{https://doi.org/10.1007/s00440-019-00907-3}.

\bibitem[Kabluchko et~al.(2019{\natexlab{b}})Kabluchko, Temesvari, and
  Th\"{a}le]{beta_polytopes_temesvari}
Z.~Kabluchko, D.~Temesvari, and C.~Th\"{a}le.
\newblock Expected intrinsic volumes and facet numbers of random
  beta-polytopes.
\newblock \emph{Math. Nachr.}, 292\penalty0 (1):\penalty0 79--105,
  2019{\natexlab{b}}.
\newblock \doi{10.1002/mana.201700255}.
\newblock URL \url{https://doi.org/10.1002/mana.201700255}.

\bibitem[Kingman(1969)]{kingman_secants}
J.~F.~C. Kingman.
\newblock Random secants of a convex body.
\newblock \emph{J. Appl. Probability}, 6:\penalty0 660--672, 1969.
\newblock \doi{10.1017/s0021900200026693}.
\newblock URL \url{https://doi.org/10.1017/s0021900200026693}.

\bibitem[McMullen(1975)]{mcmullen}
P.~McMullen.
\newblock Non-linear angle-sum relations for polyhedral cones and polytopes.
\newblock \emph{Math. Proc. Cambridge Philos. Soc.}, 78\penalty0 (2):\penalty0
  247--261, 1975.
\newblock \doi{10.1017/S0305004100051665}.
\newblock URL \url{https://doi.org/10.1017/S0305004100051665}.

\bibitem[McMullen(1986)]{mcmullen_polyhedra}
P.~McMullen.
\newblock Angle-sum relations for polyhedral sets.
\newblock \emph{Mathematika}, 33\penalty0 (2):\penalty0 173--188 (1987), 1986.
\newblock \doi{10.1112/S0025579300011165}.
\newblock URL \url{https://doi.org/10.1112/S0025579300011165}.

\bibitem[Meijering(1953)]{meijering}
J.~L. Meijering.
\newblock Inferface area, edge length, and number of vertices in crystal
  aggregates with random nucleation.
\newblock \emph{Philips Res. Rep.}, 8:\penalty0 270--290, 1953.

\bibitem[Miles(1970)]{miles_synopsis}
R.~E. Miles.
\newblock A synopsis of ``{P}oisson flats in {E}uclidean spaces''.
\newblock \emph{Izv. Akad. Nauk Armjan. SSR Ser. Mat.}, 5\penalty0
  (3):\penalty0 263--285, 1970.

\bibitem[{Miles}(1971)]{miles}
R.E. {Miles}.
\newblock {Isotropic random simplices.}
\newblock \emph{{Adv. Appl. Probab.}}, 3:\penalty0 353--382, 1971.
\newblock \doi{10.2307/1426176}.

\bibitem[M{\o}ller(1989)]{moller}
J.~M{\o}ller.
\newblock Random tessellations in {${\bf R}^d$}.
\newblock \emph{Adv. in Appl. Probab.}, 21\penalty0 (1):\penalty0 37--73, 1989.
\newblock \doi{10.2307/1427197}.
\newblock URL \url{https://doi.org/10.2307/1427197}.

\bibitem[M{\o}ller(1994)]{moller_book}
J.~M{\o}ller.
\newblock \emph{Lectures on random {V}orono\u{\i} tessellations}, volume~87 of
  \emph{Lecture Notes in Statistics}.
\newblock Springer-Verlag, New York, 1994.
\newblock \doi{10.1007/978-1-4612-2652-9}.
\newblock URL \url{https://doi.org/10.1007/978-1-4612-2652-9}.

\bibitem[Reitzner(2005)]{ReitznerCombinatorialStructure}
M.~Reitzner.
\newblock The combinatorial structure of random polytopes.
\newblock \emph{Adv. Math.}, 191\penalty0 (1):\penalty0 178--208, 2005.
\newblock \doi{10.1016/j.aim.2004.03.006}.

\bibitem[{R\'enyi} and {Sulanke}(1963)]{renyi_sulanke1}
A.~{R\'enyi} and R.~{Sulanke}.
\newblock {\"Uber die konvexe H\"ulle von $n$ zuf\"allig gew\"ahlten Punkten.}
\newblock \emph{{Z. Wahrscheinlichkeitstheor. Verw. Geb.}}, 2:\penalty0 75--84,
  1963.
\newblock \doi{10.1007/BF00535300}.

\bibitem[{R\'enyi} and {Sulanke}(1964)]{renyi_sulanke2}
A.~{R\'enyi} and R.~{Sulanke}.
\newblock {\"Uber die konvexe H\"ulle von $n$ zuf\"allig gew\"ahlten Punkten.
  II.}
\newblock \emph{{Z. Wahrscheinlichkeitstheor. Verw. Geb.}}, 3:\penalty0
  138--147, 1964.
\newblock \doi{10.1007/BF00535973}.

\bibitem[{Ruben} and {Miles}(1980)]{ruben_miles}
H.~{Ruben} and R.E. {Miles}.
\newblock {A canonical decomposition of the probability measure of sets of
  isotropic random points in $\mathbb R^n$.}
\newblock \emph{{J. Multivariate Anal.}}, 10:\penalty0 1--18, 1980.
\newblock \doi{10.1016/0047-259X(80)90077-9}.

\bibitem[Schneider(2008)]{schneider_polytopes}
R.~Schneider.
\newblock Recent results on random polytopes.
\newblock \emph{Boll. Unione Mat. Ital. (9)}, 1\penalty0 (1):\penalty0 17--39,
  2008.

\bibitem[Schneider and Weil(2008)]{SW08}
R.~Schneider and W.~Weil.
\newblock \emph{Stochastic and {I}ntegral {G}eometry}.
\newblock Probability and its Applications. Springer--Verlag, Berlin, 2008.

\bibitem[Wieacker(1978)]{wieacker_dipl}
J.~A. Wieacker.
\newblock Einige {P}robleme der polyedrischen {A}pproximation.
\newblock Diplomarbeit, Freiburg i.\ Br.\, 1978.

\end{thebibliography}
\bibliographystyle{plainnat}

\end{document}